\documentclass[12pt,a4paper,reqno,oneside]{amsart}
\usepackage{t1enc}
\usepackage{amssymb}
\usepackage[mathscr]{euscript}
\usepackage{color}
\usepackage{float}
\usepackage{graphicx}
\usepackage{url}
\usepackage{caption}
\usepackage{subcaption}
\usepackage{textcomp}
\usepackage{gensymb}%degree

\usepackage{multicol}

\usepackage{mathtools}
\usepackage[hidelinks]{hyperref}
\usepackage[capitalize,nameinlink]{cleveref}[0.19]

\usepackage{tikz}
\usetikzlibrary{intersections}
\usetikzlibrary{patterns}
\usetikzlibrary{shapes.geometric, arrows}

\numberwithin{equation}{section}

\newtheorem{defi}{Definition}[section]
\newtheorem{thm}[defi]{Theorem}
\newtheorem{lemm}[defi]{Lemma}
\newtheorem{rem}[defi]{Remark}

\newtheorem{cor}[defi]{Corollary}
\newtheorem{prop}[defi]{Proposition}
\newtheorem{exam}[defi]{Example}

\crefname{lemm}{lemma}{lemmas}
\crefname{cor}{corollary}{corollaries}

\newcommand{\convh}{\operatorname{convh}}

\newcommand{\HH}{\mathcal{H}}

\newcommand{\mesh}{\mathrm{mesh}} %%%%%%%%%%%%%%%%%%%%
\newcommand{\disp}{\mathrm{disp}}
\newcommand{\diam}{\mathrm{diam}}

\newcommand{\B}{\mathcal{B}}
\newcommand{\EE}{\mathcal{E}}
\newcommand{\U}{\mathcal{U}}

\newcommand{\DD}{\mathcal{D}}

\newcommand{\E}{\mathbb{E}}

\newcommand{\N}{\mathbb{N}}
\newcommand{\R}{\mathbb{R}}

\newcommand{\SSS}{\mathcal{S}}
\newcommand{\TTT}{\mathcal{T}}
\newcommand{\OO}{\mathcal{O}}
\newcommand{\I}{\mathcal{I}}
\newcommand{\J}{\mathcal{J}}

\newcommand{\FF}{\mathcal{F}}

\newcommand{\red}{\textcolor{red}}

\newcommand{\CC}{\mathcal{C}}

\newcommand{\norm}[1]{\left\|#1\right\|}
\begin{document}
%\linenumbers

\title[Minimal Convex Environmental Contours]
{Minimal Convex Environmental Contours}

\author[Å. H. Sande]{Åsmund Hausken Sande}
\address{Å. H. Sande: Department of Mathematics, University of Oslo, Moltke Moes vei 35, P.O. Box 1053 Blindern, 0316 Oslo, Norway.}
\email{aasmunhs@math.uio.no}
\author[J. S. Wind]{Johan S. Wind}
\address{J. S. Wind: Department of Mathematics, University of Oslo, Moltke Moes vei 35, P.O. Box 1053 Blindern, 0316 Oslo, Norway.}
\email{johanswi@math.uio.no}

\maketitle

\begin{center}
This Version : \today
\end{center}

\begin{abstract}
We develop a numerical method for the computation of a minimal convex and compact set, $\mathcal{B}\subset\mathbb{R}^N$, in the sense of mean width. This minimisation is constrained by the requirement that $\max_{b\in\mathcal{B}}\langle b , u\rangle\geq C(u)$ for all unit vectors $u\in S^{N-1}$ given some Lipschitz function $C$. 

This problem arises in the construction of environmental contours under the assumption of convex failure sets. Environmental contours offer descriptions of extreme environmental conditions commonly applied for reliability analysis in the early design phase of marine structures. Usually, they are applied in order to reduce the number of computationally expensive response analyses needed for reliability estimation.

We solve this problem by reformulating it as a linear programming problem. Rigorous convergence analysis is performed, both in terms of convergence of mean widths and in the sense of the Hausdorff metric. Additionally, numerical examples are provided to illustrate the presented methods.
\end{abstract}



\iffalse
\textcolor{green}{
"Middle of May" updates:
\begin{itemize}
    \item notation update ($d_{\HH_\HH}$ for sets of sets of $\R^N$)
    \item added filler text for \Cref{sec:hausdorff}, also added an needed extra result (\cref{cor:Dcompact})
    \item rewrote chapter on 2d results
    \item wrote introduction (is it too technical)
            \item cleaned up $d_H$, $d_\HH$, $d_{\HH_\HH}$
\end{itemize}}

\green{
April 14. update: modified order of results added some (preliminary?) filler text (in green)
}

Major todos:
\begin{itemize}
    \item "Basic intro to needed theory" \checkmark
    \item Numerical examples: Describe Arne-type numerical examples + generally improve (\checkmark)
\end{itemize}
Minor todos:
\begin{itemize}
    \item Refer to definitions in "basic intro to needed theory", when talking about $\epsilon$-accurate quadratures, dispersion, etc. Now we just randomly recite the definitions sometimes.
    \item "last paragraph is subject to change"
    \item Font sizes in plots?
\end{itemize}
\fi

\section{Introduction}
Environmental contours are mathematical tools applied to analyse the reliability of marine structures, often used in the early design phase of e.g.\ ships or oil platforms. They provide a summary statistic of the relevant environmental factors which reduces the number of computationally expensive response analyses needed. Due to this, environmental contours are widely used in reliability analysis of various marine structures \cite{baarholm2010combining,fontaine2013reliability,giske2018long,vanemTrend}, and is listed in the \textit{recommended practices - environmental conditions and environmental loads} document by DNV (Det Norske Veritas) \cite{veritas2000environmental}. The latter document provides algorithms for practical computation of environmental contours, with certain parts relying on heuristic methods. In this article, we provide algorithms to resolve this issue.

We consider $N$ environmental factors governed by a stochastic process in $\R^N$. A marine structure is then assumed to have a failure set $\FF \subset \R^N$ of environmental conditions it cannot safely handle. Let $\tau_\FF=\inf\{t:V_t\in\FF\}$ be the first hitting time of the failure set. Generally, an environmental contour is the boundary of a set $\B \subset \R^N$, representing safe environmental conditions the structure should withstand. Since the exact shape of the failure set $\FF$ is often unknown in the early design stages, the environmental contour is chosen to restrain the time to failure, $\tau_\FF$, for \textit{any} failure set $\FF$ not overlapping with the chosen set $\B$. Usually, this restriction comes in the form of an indirect lower bound on the \textit{return period} $\E[\tau_\FF]$. 

Note that while this formulation considers failure sets, it is a slight abuse of terminology. The definition of environmental contours is often decoupled from the structural response, effectively ignoring it. The response of a structure to environmental stress is usually considered to be stochastic, even for fixed environmental conditions. As such, the worst response may randomly occur outside of $\FF$. The formulation in terms of $\tau_\FF$ is still used in, or is equivalent to, most definitions of environmental contours, and more succinctly expresses the underlying ideas behind environmental contours. For a few examples on how these contours are connected with structural response in practise we refer to \cite{sagrilo2011long,giske2018long}. Both discuss the use of environmental contours in order to find a singular point, referred to as the \textit{design point}, around which the most critical environmental conditions lie. Based on this point, an importance sampling procedure can be carried out. It is also worth mentioning that \cite{giske2018long} also uses quantiles of the response distribution along the contour to estimate so-called \textit{characteristic extreme response}.

Several methods for defining and computing environmental contours exist, the most popular of which is the inverse first order reliability method (IFORM) developed in \cite{IFORM2008,IFORM}. For a thorough summary and comparison of different contour methods we refer to \cite{contcompare,contsummary}, and several models for the environmental variables along with their resulting contours are compared in \cite{Leira}. 

In this article, we will consider methods based on the additional assumption that all possible failure sets are convex. Several approaches within this setting exist, such as the method based on direct Monte Carlo simulation developed in \cite{firstaltcontour,altcontour}. This method was recently included alongside IFORM in the aforementioned \textit{best practices} document for environmental loads by DNV \cite{veritas2000environmental}, and has been further studied and extended in several other papers. For example, convexity properties of the resulting contours, as well as the inclusion of omission factors, was studied in \cite{convcont} and the extension of \textit{buffered} contours was introduced in \cite{dahl2018buffered}. 
%There are also other similar approaches, such as \cite{mackay2023model}, where a de-correlation procedure is combined by separately modelling different projections of the environmental factors, 
There are also other similar approaches, such as \cite{huseby2023AR1,mackay2023model,vanem2023analysing}, which consider serial correlation of $V$, and \cite{sande2023convex}, which extends the theory of \cite{altcontour} to a serially dependent and non-stationary setting. The common factor of these approaches, is that convexity of the failure sets allows the requirements on $\B$ to be stated as minimum outreach requirements. Specifically, all the requirements are lower bounds on the outreach function of $\B$. The outreach is defined as $B(\B,u) = \max_{b\in\B}\langle b , u\rangle$ for directions $u$ on the standard hypersphere $S^{N-1}$.

The requirements can be succinctly stated as $B(\B,u)\geq C(u)$ for all $u\in S^{N-1}$, where $C$ is some function computed to ensure the desired restriction of the failure time, usually a lower bound on the return period, $\E[\tau_\FF]$. If this inequality holds, we say that $\partial\B$ is a \textit{valid} contour. The goal is then to find a valid contour which is minimal in some sense. Whenever there exists a $\B$ satisfying $B(\B,u) = C(u)$, that $\B$ is trivially optimal. When this is the case, we will refer to $\partial\B$ as a \textit{proper} contour. However, a proper contour does not always exist. In \cite{convcont}, they discuss how estimation errors in $C$ can lead to the absence of proper contours. To correct this they suggest inflating $C$ by adding a sufficiently large constant, ensuring that the resulting contour is valid, as well as proper with respect to the inflated $C$. Similarly, in \cite{voronoi}, examples of distributions for $V$ are given which allow for no proper contours, regardless of estimation errors. The article also proposes the following method for constructing valid contours in this setting. They construct an invalid contour based on $C$, and then extend it in all directions in order to ensure validity. However, both these methods fail to establish \textit{minimal} valid contours in a general sense. The contour is intended to represent the most extreme safe conditions for a structure. However, if the contour is made too big it imposes stronger constraints on the class of structures it applies to, thereby shrinking this class. As such it is of interest to construct contours which are minimal in some sense, which would allow us to apply the contour, and the resulting restrictions on $\tau_\FF$, to as wide a class as possible.

Our goal in this article is to present a method for constructing valid contours that are minimal in the sense of \textit{mean width} (a generalisation of the perimeter length to more than two dimensions). We show how to solve this by casting it as a linear programming problem. We prove bounds on the sub-optimality incurred by discretisation of the continuous problem, and give convergence results to ensure that our method can find solutions with mean width arbitrarily close to the optimum.

{
In \Cref{sec:mainresultin2d}, we provide a simple explanation of our main goal and results for the two-dimensional case. To proceed, we first cover some necessary definitions and results in \Cref{sec:setup}, before moving on to \Cref{sec:mincont}, where we present our main results in a general setting. Specifically, we prove that our method provides arbitrarily near-optimal solutions. In order to illustrate our results, we present numerical examples in \Cref{sec:examples}. Next, we present several results that guarantee convergence of our method in the Hausdorff metric in \Cref{sec:hausdorff}. As it turns out that our method can be simplified and improved in the two-dimensional case, we briefly present the improved method in \Cref{sec:bonus2d}. Finally, as part of our method, we use quadratures for numerical integration with certain properties. We present some simple constructions for generating such quadratures in \Cref{sec:quadraturecomp}.
}

\section{Main results in two dimensions}\label{sec:mainresultin2d}
We state our main results in two dimensions here, and wait until the necessary setup has been made before stating the general result in \Cref{thm:DPepsinCPeps_discrete}. This is only to give a simplified statement, our method and proofs are developed for general dimensions.

We are interested in finding the convex, compact shape $\B$ with the smallest perimeter, which has a given outreach in each direction. Specifically, the data to our problem is a $L_C$-Lipschitz function from the unit circle to the real numbers, $C \colon S^1 \to \R$. We denote the maximum absolute outreach requirement $\norm{C}_\infty \coloneqq \sup_{u \in S^1} |C(u)|$. Formally, the outreach requirement on $\B$ is this: For all directions $u \in S^1$, there exists some $p \in \B$ such that $\langle p, u\rangle \ge C(u)$.

Numerically, we only access $C$ through a finite number of samples. In two dimensions, we can sample $m$ evenly spaced directions $\{u_i\}_{i=1}^m$. If we restrict $\B$ to polygons with sides perpendicular to the directions $\{u_i\}_{i=1}^m$, and only consider outreach requirements in those directions, we can formulate the resulting problem as a linear programming problem \eqref{lpprob}. In two dimensions there is also a more efficient formulation \eqref{lpprob_2d1}. These linear programs can be solved efficiently using standard techniques, giving an optimal solution $\widetilde \B$ to the discretised problem.

Our theory shows that $\widetilde \B$ is nearly an optimal solution to the continuous problem. Specifically, if we inflate $\widetilde \B$ (in the sense of \Cref{lem:contour_extension}) by a small amount $\mathcal{O}\left(\frac{1}{m}(L_C+\norm{C}_\infty)\right)$, it is guaranteed to satisfy the outreach requirements in {\it all} directions. Furthermore, the inflated $\widetilde \B$ has a perimeter at most $\mathcal{O}\left(\frac{1}{m}(L_C+\norm{C}_\infty)\right)$ more than the optimal perimeter over {\it all} convex, compact shapes satisfying the outreach requirements.

%(an optimal shape exists by \Cref{thm:CP_sol_exist}).
When the number of dimensions is more than two, the perimeter is generalised to the mean width (\Cref{def:meanwidth}), and evenly spaced sample directions are generalised to $\epsilon$-accurate quadratures (\Cref{def:epsilon_accurate}) with low dispersion (\cref{def:dispersion}). We define these concepts in the next section.

\section{basic intro to needed theory}\label{sec:setup}
 In this section we will define the main functions we will need throughout the article.

\subsection{Standard Notation}
As in \Cref{sec:mainresultin2d} we denote by $\|\cdot\|$ and $\langle\cdot,\cdot\rangle$ the canonical norm and inner product on $\R^N$. Furthermore, the hypersphere in $\R^N$ is defined by $S^{N-1}=\{u\in\R^N:\|u\|=1\}$. We will also need the uniform probability measure on $S^{N-1}$, denoted by $\sigma$.

\subsection{Convexity}
A key concept when dealing with convex sets is their outreach.

\begin{defi}\label{def:outreach}
    For any set $\B\subset\R^N$ we define the \textit{outreach function} of $\B$ as
    $$B(\B,u)=\sup_{b\in\B}\langle b , u\rangle.$$
    This function is also commonly referred to as the \textit{support function} of $\B$.
\end{defi}

A closely related concept is the idea of a hyperplane.

\begin{defi}\label{def:half_hyper}
    We define the \textit{hyperplane} for some $ b\in\R $,  $ u\in S^{N-1} $ as
    \begin{equation*}
    \Pi(u,b)   = \{v\in\R^N:\langle u,v\rangle  =  b \}.
    \end{equation*}
    
    We further define the \textit{half-spaces}
    \begin{align*}
    	\begin{split}
    		\Pi^-(u,b) &=  \{v\in\R^N:\langle u,v\rangle  \leq  b\},\\
    		\Pi^+(u,b) &=  \{v\in\R^N:\langle u,v\rangle  \geq  b\}.
    	\end{split}
    \end{align*}
\end{defi}

Consider then $b\in\R$, $u\in S^{N-1}$, and some non-empty, convex, and compact $\B\subset\R^N$. Since $\B$ is compact we have that $B(\B,u)$ is finite. We can even guarantee some regularity of $B(\B,\cdot)$ by the following result.

\begin{prop}\label{prop:B_lipz}
  Let $\mathcal{B} \subset \R^N$ be a non-empty, convex, and compact set. Denote the maximum radius $R \coloneqq \sup_{p \in \mathcal{B}} \norm{p}_2$. Then $B(\mathcal{B},u) = \max_{p \in \mathcal{B}} \langle p, u\rangle$ is $R$-Lipschitz as a function of $u \in S^{N-1}$. Note also $|B(\mathcal{B},u)| \le R$.
\end{prop}
\begin{proof}
  For fixed $p \in \mathcal{B}$, the function $\langle  \cdot,p\rangle$ is $\norm{p}_2$-Lipschitz, hence $R$-Lipschitz. Since $B(\mathcal{B},u)$ is the supremum of $R$-Lipschitz functions, it is itself $R$-Lipschitz.
\end{proof}

\Cref{prop:B_lipz} guarantees that $B(\B,\cdot)$ is integrable which allows us to define our key measure of size for convex compact sets.

\begin{defi}\label{def:meanwidth}
    The width of a a non-empty, convex, and compact set $\B\subset\R^N$ along a vector $u\in S^{N-1}$ can be written as $B(\B,u)-B(\B,-u)$. As such, we define the \textit{mean width} of $\B$ by
    $$\int_{S^{N-1}}(B(\B,u)-B(\B,-u))d\sigma(u).$$
    Note that this equals
    $$2\int_{S^{N-1}}B(\B,u)d\sigma(u).$$
\end{defi}
For convex shapes in two dimensions, the mean width is equal to the perimeter divided by $\pi$.

Lastly, if $B(\B,u)\leq b$ we must also have $\B\subset \Pi^-(u,b)$. This leads to the following unique representation of compact convex sets. This result is a special case of Theorem 18.8 in \cite{rockafellar1997convex}.

\begin{prop}\label{prop:support_rep}
    Let $\B\subset\R^N$ be convex and compact, we then have
    $$\B=\bigcap_{u\in S^{N-1}}\Pi^-(u,B(\B,u)).$$
\end{prop}

\subsection{Numerical Definitions}

In order to compute the mean width our convex sets we will need to employ a numerical integration method providing universal bounds on the integration error. To address this we introduce the following.

\begin{defi}\label{def:epsilon_accurate}
  We say a set of points and weights $\{(u_i,w_i)\}_{i=1}^m \subset S^{N-1} \times [0,\infty)$ is an \it{$\epsilon$-accurate} quadrature, if the following holds for all $L$-Lipschitz continuous functions $f \colon S^{N-1} \to \R$. For all functions $f$ satisfying $|f(u)-f(v)| \le L \norm{u-v}_2$ for all $u,v \in S^{N-1}$, we have
  \begin{align}\label{eq:epsilon_accurate_def}
    \left|\int_S f(u) \,d\sigma(u) - \sum_{i=1}^m f(u_i) w_i\right| \le \epsilon (L+\norm{f}_\infty).
  \end{align}
\end{defi}

In order to control this error we will also need a related concept characterising the spread of some finite $\SSS\subset S^{N-1}$.

\begin{defi}\label{def:dispersion}
    We define the \textit{dispersion} of a set $\SSS\subset S^{N-1}$ by
    $$\disp(\SSS)=\sup_{u\in S^{N-1}}\inf_{v\in\SSS}\|u-v\|.$$
\end{defi}

\begin{rem}
    In order to carry out the numerical integration necessary in this paper, we consider a general grid $\SSS=\{u_i\}_{i=1}^m\subset S^{N-1}$ and set of weights $W=\{w_i\}_{i=1}^m$, with some constraints on $\disp(\SSS)$ and the accuracy of $\{(u_i,w_i)\}_{i=1}^m$. However, we also give specific constructions of $\epsilon$-accurate quadratures in \Cref{sec:quadraturecomp}.
\end{rem}

\section{Minimal Valid Contours}\label{sec:mincont}

In this section, we aim to compute minimal valid contours in the sense of mean width. To set the scene. We will start by defining relevant concepts about our original and approximating discrete problem, for then to present several results guaranteeing control over estimation errors.

\subsection{The Continuous Problem}

The main problem we want to solve can be precisely formulated as follows.

\begin{align}\label{cpprob}
	\text{minimise }   \quad\quad\quad & \int_{S^{N-1}} B(\B,u) d\sigma(u)& \\
	\text{subject to } \quad\quad\quad & B(\B,\cdot)\geq C(\cdot), & \label{cpconstr}\\
	& \B \text{ convex and compact.} &\nonumber
\end{align}

To examine this problem, we introduce the following notation.

\begin{defi}\label{def:cpdef}
    We denote the set of feasible solutions to the continuous problem \eqref{cpprob} by
    $$\CC^\infty = \left\{ \B\subset\R^N \mid \B \text{ convex and compact, } B(\B,\cdot)\geq C(\cdot) \right\}.$$
    The optimal value of the objective function is denoted by
    $$V^{CP} = \inf \left\{ \int_{S^{N-1}} B(\B,u) d\sigma(u) ,\, \B\in\CC^\infty\right\}.$$
    Lastly, we define the $\gamma$-near optimal solution space by
    $$\CC^\gamma = \left\{ \B\in\CC^\infty \Big| \int_{S^{N-1}} B(\B,u) d\sigma(u) - V^{CP} \leq \gamma\right\}.$$
\end{defi}

It turns out that this problem has an optimal solution under our running assumption that $C$ is continuous. This result is proven through consideration of the Hausdorff metric, as such the proof is relegated to \Cref{thm:CP_sol_exist} in \Cref{sec:hausdorff}.

\subsection{The Discrete Problem}

In what follows, we will consider a way of approximating optimal contours. To achieve this, we will need to employ numerical methods, which necessitate discretisation. Therefore, we also consider valid environmental contours with respect to some (usually finite) sub-collection of unit vectors $ \SSS\subseteq S^{N-1} $.

We will say that $ \partial\B $ is $ (\SSS,C) $\textit{-valid} if $ \B $ is convex, compact, and for all $ u\in\SSS $ we have $ B(\B,u)\geq C(u) $.

If we take any $ \SSS=\{u_i\}_{i=1}^m $ and a $ (\SSS,C) $-valid contour $ \partial\B $, then for every $ u_i\in\SSS $ there must be some $ p_i\in \B$ such that $ \langle p_i,u_i\rangle=B(\B,u_i) $, which implies
\begin{equation}\label{lp1}
	\langle p_i,u_i\rangle\geq C(u_i) \text{ for }  i=1,2,\dots,m.
\end{equation}
Furthermore, since for all $ i $ we have $ p_i\in\B $, we must also have 
\begin{equation}\label{lp2}
\langle p_i,u_j\rangle\leq B(\B,u_j)  \text{ for }  i,j=1,2,\dots,m.
\end{equation}
Conversely, assume we have a set of points $ \{p_i\}_{i=1}^m $ satisfying \eqref{lp1} and \eqref{lp2} for some convex and compact set $ \B $. By \eqref{lp2} we know that $ \{p_i\}_{i=1}^m\subset\B $, which implies by \eqref{lp1} that $ B(\B,u_i)\geq C(u_i) $ for all $ i $, which means that $ \partial\B $ is $ (\SSS,C) $-valid.

Consequently, there is a correspondence between $ (\SSS,C) $-valid contours and sets of points, $ \{p_i\}_{i=1}^m $. Hence, we consider the following linear programming problem. Note that in order to approximate optimisation in mean width, we will refer to this as the linear program based on $(\SSS,W)$ where $W=(w_i)_{i=1}^m$ is a set of weights such that $\{(u_i,w_i)\}_{i=1}^m$ forms an $\epsilon$-accurate quadrature for some $\epsilon\in\R_+$.
\begin{align}\label{lpprob}
	\text{minimise }   \quad\quad\quad & \sum_{i=1}^m w_iB_i& \\
	\text{subject to } \quad\quad\quad & \langle p_i,u_i\rangle\geq C(u_i) 
	& i=1,2,\dots,m \nonumber\\
	& \langle p_i,u_j\rangle\leq B_j    & i,j=1,2,\dots,m \nonumber\\
	&  p_i \in \R^N    & i=1,2,\dots,m \nonumber\\
    &  u_i \in S^{N-1}    & i=1,2,\dots,m \nonumber\\
	&  C(u_i),B_i \in \R    & i=1,2,\dots,m \nonumber
\end{align}

We then note two facts about this problem. Firstly, the values $p_i = \|C\|_\infty u_i$, $ B_i = \|C\|_\infty $ for all $ i $ with $\|C\|_\infty = \max_{u\in S^{N-1}}|C(u)|$, satisfies the constraints and provides a feasible solution. Secondly, since $ \sum_{i=1}^m w_i B_i \geq \sum_{i=1}^m w_i C(u_i) > -\infty $, the objective function is bounded. Combining these facts we know that the problem must have at least one optimal solution, which yields a minimal $(\SSS,\CC)$-valid contour by the either of the following two constructions.

\begin{prop}\label{prop:Bstar}
	Consider the linear programming problem \eqref{lpprob}, with an optimal solution $ \left(\left(p^*_i\right)_{i=1}^m , \, \left(B^*_i\right)_{i=1}^m \right)$. If $\convh(\cdot)$ denotes the convex hull we have that
	\begin{equation*}
	\B^*=  \convh\left(\left\{p^*_i\right\}_{i=1}^m\right),
	\end{equation*}
	defines a $(\SSS,C)$-valid contour with $ B(\B^*,u_i)=B_i $ for all $ i $.
	
	\begin{proof}
		Firstly, we note that every $ b\in\B^* $ is a convex combination of the $p^*_i$s. This means $ b $ has the representation $ b=\sum_{i=1}^m a_i(b)p_i $ where $ \sum_{i=1}^m a_i(b) = 1,\, a_i\geq 0  $, which implies
		\begin{align*}
			B(\B^*,u_i)
			&=    \max_{b\in\B^*}\langle b,u_i\rangle\\
			&=    \max_{b\in\B^*}\sum_{k=1}^m a_k(b)\langle p_k,u_i\rangle\\
			&\leq \max_j\langle p_j,u_i\rangle\max_{b\in\B^*}\sum_{k=1}^m a_k(b)\\
			&=    \max_j\langle p_j,u_i\rangle.
		\end{align*}
		Conversely, since $ \left\{p^*_i\right\}_{i=1}^m\subseteq \B^*$ we have $\max_j\langle p_j,u_i\rangle\leq B(\B^*,u_i) $ which implies $  B(\B^*,u_i) =\max_j\langle p_j,u_i\rangle$.
		
		As a consequence we get
		$$ B(\B^*,u_i)\geq\langle p_i,u_i\rangle\geq C(u_i). $$
		As the convex hull of a finite number of points, $\B^*$, is compact and convex. These facts make $\partial\B^*$ a $(\SSS,C)$-valid contour.
	\end{proof}
\end{prop}

\begin{cor}\label{cor:Bprime}
	Consider the linear programming problem \eqref{lpprob}, with an optimal solution $ \left(p_i\right)_{i=1}^m , \, \left(B_i\right)_{i=1}^m $. We then have that
	\begin{equation*}
		\B'=  \bigcap_{i=1}^m\Pi^-(u_i,B_i),
	\end{equation*}
	defines a $(\SSS,C)$-valid contour with $B(\B',u_i) = B_i$ for all $i$.
	
	\begin{proof}
		We first note that since $ \langle p_i,u_j\rangle\leq B_j $ for all $ i,j$, we must have $ p_i \in \B'$ for all $ i $. This implies that when $ \B^* $ is as defined in \Cref{prop:Bstar}, we have $ \B^*\subseteq\B' $ since $\B'$ is convex. This immediately implies that $ \partial\B' $ is $ (\SSS,C) $-valid.
	\end{proof}
\end{cor}

Similarly to \Cref{def:cpdef}, we consider the analogous concepts for this discrete problem, which will aid our comparison between the continuous problem \eqref{cpprob} and our discrete approximation \eqref{lpprob}.

\begin{defi}\label{dpdef}
    We denote the set of \textit{valid} solutions to the discrete problem \eqref{lpprob} based on $ (\SSS,W) $ by
    $$\DD^\infty(\SSS,W) = \left\{ \B\subset\R^N \mid \B \text{ convex and compact, } B(\B,u)\geq C(u) \text{ for all } u\in\SSS \right\}.$$
    The optimal value of the objective function is denoted by
    $$V^{DP}(\SSS,W) = \min \left\{ \sum_{i=1}^m w_iB(\B,u_i) ,\, \B\in\DD^\infty\right\}.$$
    Lastly, we define the $\gamma$-near optimal solution space by
    $$\DD^\gamma(\SSS,W) = \left\{ \B\in\CC^\infty \Big| \sum_{i=1}^m w_iB(\B,u_i) - V^{DP} \leq \gamma\right\}.$$
    We will usually omit the dependence on $(\SSS,W)$ whenever the meaning is clear or otherwise superfluous.
\end{defi}

\begin{rem}\label{rem:BstarinBprime}
  If we consider any $\B\in\DD^\infty$ then we know from previous arguments that there exists a feasible solution $\left(\left(p_i\right)_{i=1}^m , \, \left(B_i\right)_{i=1}^m\right)$ with $\left\{p_i\right\}_{i=1}^m\subset\B$ and $B_i=B(\B,u_i)$. Consequently, if $B^*$ and $B'$ are defined as in \Cref{prop:Bstar} and \Cref{cor:Bprime}, we see that $\B^*\subseteq\B\subseteq\B'$, which means that the constructions of $\B^*$ and $\B'$ provide lower and upper bounds on all sets in $\DD^\infty$.
\end{rem}

\subsection{Convergence and Near-Optimality}

With these definitions established, we can more accurately state the goal of this chapter. We first aim to show how one can construct $(S^{N-1},C)$-valid contours from any $\B\in\DD^\gamma$ for some $\gamma\geq0$. Furthermore, we will prove that the optimal value of the discrete problem, $V^{DP}$, can arbitrarily well approximate $V^{CP}$. Using this, we get explicit upper bounds on the near-optimality of the constructed $(S^{N-1},C)$-valid contours.

In order to control the near-optimality of solutions to our discrete problems, we need to consider two issues. The first problem we will tackle is the fact that a $(\SSS,C)$-valid contour is not necessarily $(S^{N-1},C)$-valid. To amend this, we will consider a method for inflating contours to ensure their validity. Secondly, we will need to correct for the fact that \eqref{lpprob} optimises for an approximation of mean width, this can be handled by explicitly including the numerical error from discrete integration of the mean width.

To construct valid contours from our discrete approximation, we will first need a bound on how much a contour $\partial\B $ with $ \B\in\DD^\infty$ can violate the constraint of $B(\B,\cdot)\geq C(\cdot)$.

\begin{lemm}\label{lemm:errorbound_geqC}
  Fix some $\SSS\subset S^{N-1}$ with $\delta = \disp(\SSS)$ and let $\B \in \DD^\infty$. Then for all $u \in S^{N-1}$, we have
    $$C(u)-\delta(L_C+R) \leq B(\B,u),$$
    where $L_C$ is the Lipschitz constant of $C$, and $R=\max_{p\in\B} \|p\|$.
    
    \begin{proof}
        Consider any $u\in S^{N-1} $ and pick some $v\in\SSS$ such that $\|u-v\|\leq\delta$. This immediately yields by the Lipschitz continuity of $C$ that $ |C(u)-C(v)|\leq \delta L_C $. We then choose some $p_v\in\partial\B$ such that $\langle p_v,v\rangle=B(\B,v)\geq C(v)$ and get
        \begin{align*}
            B(\B,u)
            &\geq \langle p_v,u\rangle\\
            & =   \langle p_v,v\rangle + \langle p_v,u-v\rangle\\
            &\geq C(v)-\delta \|p_v\| \\
            &\geq  C(u)-\delta L_C - \delta R \\
            &= C(u) - \delta (L_C + R) .
        \end{align*}
    \end{proof}
\end{lemm}

With this result, we can quantify how far our $(\SSS,C)$-valid contours are from being $(S^{N-1},C)$-valid. The main idea in constructing $(S^{N-1},C)$-valid contours is to use the bound of \Cref{lemm:errorbound_geqC}, and then use following result to inflate the contour.

\begin{lemm}\label{lem:contour_extension}
    Assume we have a subset $\SSS\subseteq S^{N-1} $ and a convex and compact set  $ \B $. If we then define 
    $$\B^e=\bigcap_{ v \in \SSS } \Pi^-\left( v,B(\B,v) + e \right),$$
    we have that $B(\B^e,u)\geq B(\B,u)+e$ for all $u\in S^{N-1}$, and $B(\B^e,v)= B(\B,v)+e$ for all $v\in \SSS$.
    \begin{proof}
        For any $u\in S^{N-1}$ there exists a $ b(u)\in  \B$ such that $\langle  b(u),u \rangle = B(\B,u)$, we then have immediately that $\langle  b(u)+eu ,u\rangle = B(\B,u)+e$. Furthermore, since $ b(u)\in\B$, we have for every $v\in\SSS,\,u\in S^{N-1}$ that $\langle  b(u),v \rangle \leq B(\B,v)$ which further yields
        $$\langle  b(u)+eu,v \rangle \leq B(\B,v) + e\langle u, v \rangle\leq B(\B,v) + e.$$
        As a consequence, we must have $b(u)+eu\in\B^e$, which implies $B(\B^e,u)\geq B(\B,u)+e$ for all $u\in S^{N-1}$. Lastly, by definition of $\B^e$, we have $B(\B^e,v)\leq B(\B,v)+e$ for any $v\in\SSS$, which coupled with the previous inequality implies that $B(\B^e,v) = B(\B,v)+e$ for all $v\in\SSS$.
    \end{proof}
\end{lemm}

These results imply that for any $\gamma\geq0,\,\B\in\DD^\gamma$, we may inflate the contour in order to guarantee that the resulting $\B^e$ provides a $(S^{N-1},C)$-valid contour. Note that this construction depends on $\max_{p\in\B}\|p\|$. In order to extend these results and guarantee universal bounds on the necessary inflation, we will need the following results which limit the size of $\max_{p\in\B}\|p\|$ for all $\B$ in $\DD^\gamma$. However, to achieve this bound we will first need a small technical computation.

\begin{lemm}\label{lem:Id}
  For any $v \in S^{N-1}$, we have
  $$\int_{S^{N-1}}\langle v,u\rangle^+ d\sigma(u) \ge \frac{1}{3\sqrt{N}},$$
  where $(\cdot)^+$ denotes $\max(\cdot,0)$.
\end{lemm}
\begin{proof}
  We rewrite as an expectation of absolute values as follows
  \begin{align*}
    \int_{S^{N-1}}\langle v,u\rangle^+ d\sigma(u) &= \frac{1}{2}\mathbb{E}_{u \in \mathrm{Unif}(S^{N-1})}\left(|\langle v,u\rangle|\right)\\
                                                  &= \frac{\mathbb{E}_{z \in \mathcal{N}(0,I_N)}\norm{z}\, \mathbb{E}_{u \in \mathrm{Unif}(S^{N-1})}\left(|\langle v,u\rangle|\right)}{2\mathbb{E}_{z \in \mathcal{N}(0,I_N)}\norm{z}}.
  \end{align*}
  Noting that the multivariate normal distribution has uniformly random direction, we may calculate the numerator as
  \begin{align*}
    &\mathbb{E}_{z \in \mathcal{N}(0,I_N)}\norm{z}\, \mathbb{E}_{u \in \mathrm{Unif}(S^{N-1})}\left(|\langle v,u\rangle|\right) = \\
    &\mathbb{E}_{z \in \mathcal{N}(0,I_N)}\left(|\langle v,z\rangle|\right) = \mathbb{E}_{z_1 \in \mathcal{N}(0,1)}|z_1| = \sqrt\frac{2}{\pi}.
  \end{align*}
  The denominator can be bounded by $2\mathbb{E}_{z \in \mathcal{N}(0,I_N)}\norm{z} \le 2\sqrt{\mathbb{E}_{z \in \mathcal{N}(0,I_N)}\norm{z}^2} = 2\sqrt N$.
  In total, we get
  \begin{align*}
    \int_{S^{N-1}}\langle v,u\rangle^+ d\sigma(u) &\ge \frac{\sqrt\frac{2}{\pi}}{2\sqrt N} > \frac{1}{3\sqrt N}.
  \end{align*}
\end{proof}

\begin{lemm}\label{lem:meanwidth>diam}
    Let $\B$ be a compact set satisfying $\max_{p,q\in\B}\|p-q\|\geq K $ for some $K\in\R^+$. We then have that the halved mean width of $\B$ satisfies
    $$
    \int_{S^{N-1}}B(\B,u)d\sigma(u)
    \geq
    {\frac{K}{3\sqrt N}}.
    $$

    \begin{proof}
      We start by picking {$p_0, p \in \B$} such that $\|p-p_0\| \geq K$ and note that 
        $$B(\B,u) \geq \max(\langle p,u\rangle,\langle p_0,u\rangle)
        =\langle p-p_0,u\rangle^+ +\langle p_0,u\rangle.$$
        If we then note that $\int_{S^{N-1}}\langle p_0,u\rangle d\sigma(u) = 0$, we can define $v=\frac{p-p_0}{|p-p_0|}$ to get
        $$
        \int_{S^{N-1}}B(\B,u)d\sigma(u)
        \geq \int_{S^{N-1}} \left(\langle p-p_0,u\rangle^+ +\langle p_0,u\rangle\right) d\sigma(u)
        = K\int_{S^{N-1}}\langle v,u\rangle^+ d\sigma(u).
        $$
        {By \Cref{lem:Id} this implies the desired result.}
    \end{proof}
\end{lemm}

With these results, we can now universally bound $\max_{p\in\B}\|p\|$ for all $\B$ in $\DD^\gamma$, which will also provide a universal bound on the amount of inflation needed in \Cref{lem:contour_extension} to yield a $(S^{N-1},C)$-valid contour.

\begin{lemm}\label{lem:Bmax<CmaxK}
    Assume that $ \{(u_i,w_i)\}_{i=1}^m\subset \SSS\times [0,\infty) $ is an $\epsilon$-accurate quadrature with $\delta \coloneqq \disp(\SSS)$ satisfying $\epsilon,\delta \le \frac{1}{10\sqrt N}$. Furthermore, let $\B\in\DD^\gamma$, then
    $$
    \max_{p\in\B}\|p\| \leq 12\sqrt N(\|C\|_\infty+\gamma).
    $$

    \begin{proof}
      First, we pick some $p_{max}\in\partial\B$ such that $\|p_{max}\|=\max_{p\in\B}\|p\|$. If $\|p_{max}\|=0$ then the desired bound holds trivially. As such we assume, without loss of generality, that $\|p_{max}\|>0$ and define $w={p_{max}}/{\|p_{max}\|}$. We then consider some $ v\in\SSS $ such that $\|(-w)-v\|\leq\delta$ and note that there must be some $p_0\in\partial\B$ such that $\langle p_0,v\rangle=B(\B,v)\geq C(v) \ge -\|C\|_\infty$. This means that 
        \begin{align*}
            \langle p_0,w\rangle
            &=
            \langle p_0,w+v\rangle-\langle p_0,v\rangle\\
            &\leq 
            \delta\|p_0\|+\|C\|_\infty\\
            &\leq
            \delta\|p_{max}\|+\|C\|_\infty,
        \end{align*}
        which further yields
        $$ 
        \|p_{max}-p_0\|
        \geq
        \langle p_{max}-p_0,w\rangle
        \geq
        \|p_{max}\|(1-\delta)-\|C\|_\infty.
        $$
        This means that 
        $$\max_{p,q\in\B}\|p-q\|\geq \|p_{max}\|(1-\delta)-\|C\|_\infty,$$
        which, by \Cref{lem:meanwidth>diam}, gives
        $$
        \int_{S^{N-1}}B(\B,u)d\sigma(u)
        \geq
        \frac{\|p_{max}\|(1-\delta)-\|C\|_\infty}{3\sqrt N}.
        $$
        On the other hand, since $ p_i=\|C\|_\infty u_i $, $ B_i = \|C\|_\infty $ for all $ i $ is a feasible solution of the linear program, we must also have
        %
        %\begin{align*}
        %    \sum_{i=1}^m w_i B(\B,u_i) 
        %    &\leq V^{DP}+\gamma\\
        %    &\leq \sum_{i=1}^m w_i C(u_i) + \gamma \\
        %    &\leq (1+\epsilon) \|C\|_\infty + \gamma.
        %\end{align*}
        $$\sum_{i=1}^m w_i B(\B,u_i) 
            \leq V^{DP}+\gamma
            \leq \sum_{i=1}^m w_i C(u_i) + \gamma 
            \leq (1+\epsilon) \|C\|_\infty + \gamma.$$
        Finally, by \Cref{prop:B_lipz} and the definition of an $\epsilon$-accurate quadrature, we have that
        $$\left|\int_{S^{N-1}}B(\B,u)d\sigma(u)-\sum_{i=1}^m w_i B(\B,u_i) \right| \leq 2\epsilon\|p_{max}\|.$$
        Putting these facts together, we get
        \begin{align*}
            (1+\epsilon)\|C\|_\infty+\gamma
            &\geq \sum_{i=1}^m w_i B(\B,u_i)\\
            &\geq \int_{S^{N-1}}B(\B,u)d\sigma(u) - 2\epsilon\|p_{max}\|\\
            &\geq \frac{\|p_{max}\|(1-\delta)-\|C\|_\infty}{3\sqrt N} - 2\epsilon\|p_{max}\|.
        \end{align*}
        Using the assumptions $\epsilon,\delta \le \frac{1}{10\sqrt N}$ and $N \ge 2$, we get
        $$
        \max_{p\in\B}\|p\| \leq \frac{3\sqrt N((1+\epsilon)\|C\|_\infty+\gamma)+\|C\|_\infty}{1-\delta-6\sqrt N\epsilon} \le 12\sqrt N(\|C\|_\infty+\gamma).
        $$
    \end{proof}
\end{lemm}

Combining \Cref{lemm:errorbound_geqC,lem:Bmax<CmaxK,lem:contour_extension}, we can construct $(S^{N-1},C)$-valid contours from our discrete approximation. This allows us to directly compare our discrete solutions from \eqref{lpprob} to the theoretical ones of \eqref{cpprob}. However, we still need to address the error stemming from the numerical integration of $B(\B,\cdot)$. To handle this, we note that the integration error of an $\epsilon$-accurate quadrature involves $\max_{p\in\B}\|p\|$, which is bounded for $\B\in\DD^\gamma$ by \Cref{lem:Bmax<CmaxK}. To finish our preparations for the main theorem, we prove the analogous result for $\B\in\CC^\gamma$.

\begin{cor}\label{cor:Bmax<CmaxK_cont}
    Assume $\B \in \CC^\gamma$ for some $\gamma \geq 0$, we then have that 
    $$\max_{p\in\B}\|p\| \leq 4\sqrt{N}\left(\|C\|_\infty+\gamma\right), $$
  where $\|C\|_\infty = \max_{u\in S^{N-1}}|C(u)|$.
    \begin{proof}
      First, we pick some $p_{max} \in \B$ such that $\|p_{max}\| = \max_{p\in\B}\|p\|$.  If $\|p_{max}\|=0$ then the desired bound holds trivially. As such we assume, without loss of generality, that $\|p_{max}\|>0$ and define $v = p_{max}/\|p_{max}\|$. Since $\partial\B$ is $(S^{N-1},C)$-valid, there must also exist some $p_0 \in \partial\B$ such that $\langle p_0,-v\rangle \geq C(-v) \ge -\|C\|_\infty$, which implies 
        \begin{align*}
        \|p_{max}-p_0\|
        &\geq \langle p_{max}-p_0,v\rangle 
        = \langle p_{max},v\rangle+\langle p_0,-v\rangle 
        \geq \|p_{max}\|-\|C\|_\infty.
        \end{align*}
        As a consequence, we note that since $\max_{p,q\in\B}\|p-q\|\geq \|p_{max}\|-\|C\|_\infty$, \Cref{lem:meanwidth>diam} implies that
        $$\int_{S^{N-1}}B(\B,u)d\sigma(u)
        \geq
        \frac{\|p_{max}\|-\|C\|_\infty}{3\sqrt N}.$$

        Furthermore, we have that $\|C\|_\infty S^{N-1} \in \CC^\infty$ which means that $V^{CP}\leq \|C\|_\infty$ and therefore $\B \in \CC^\gamma$ implies $\int_{S^{N-1}} B(\B,u) d\sigma(u) \leq \|C\|_\infty+\gamma$.

        Combining these facts with $N \ge 1$ yields $$\|p_{max}\| \le 3\sqrt N(\|C\|_\infty+\gamma)+\|C\|_\infty \le 4\sqrt N(\|C\|_\infty+\gamma).$$
    \end{proof}
\end{cor}

With these universal bounds established, we can discuss the two main results of this chapter, which guarantee that our discrete problem \eqref{lpprob} indeed approximates \eqref{cpprob}.

\begin{thm}\label{thm:errorbound_optval}
  Fix some $\epsilon$-accurate quadrature $\{(u_i,w_i)\}_{i=1}^m$ and assume that $\SSS=\{u_i\}_{i=1}^m$ with $\delta=\disp(\SSS)$ satisfies $\epsilon,\delta \le \frac{1}{10\sqrt N}$. Then
    \begin{align*}
      V^{CP}-V^{DP} &\le 12\sqrt N \|C\|_\infty (2\epsilon+\delta)+L_C \delta,\\
      V^{DP}-V^{CP} &\le 8\sqrt N \|C\|_\infty \epsilon.
    \end{align*}
    %
    
    %Here $R^{DP}=\max_{p\in\B}$ for any compact $\B$ that solves the discrete problem in the sense that there exists $\left\{p^*_i\right\}_{i=1}^m \subseteq \B$ such that $\left( \left(p^*_i\right)_{i=1}^m , \, \left(B(\B,u_i)\right)_{i=1}^m \right)$ is an optimal solution of the linear programming problem \eqref{lpprob}. Similarly, $R^{CP}=\max_{p\in\B}$ for any optimal solution $\B$ of the continuous problem \eqref{cpprob}, note that both $R^{CP}$ and $R^{DP}$ are both uniformly bounded by \Cref{lem:Bmax<CmaxK} for sufficiently small $(\epsilon,\delta)$.
    
    \begin{proof}
        For the first statement, let $\left( \left(p^*_i\right)_{i=1}^m , \, \left(B^*_i\right)_{i=1}^m \right)$ be an optimal solution of the linear programming problem \eqref{lpprob} and choose any $\B^{DP}\in\DD^0$ with $B(\B^{DP},u_i)=B^*_i$ for $i=1,2,\dots,m$. We then denote $R^{DP}=\max_{p\in\B^{DP}}\|p\|$ and note that \Cref{prop:B_lipz} along with the definition of an $\epsilon$-accurate quadratures imply
        $$\left|\sum_{i=1}^m w_iB^*_i - \int_{S^{N-1}}B(\B^{DP},u)d\sigma(u)\right| \leq 2\epsilon R^{DP}.$$

        By \Cref{lemm:errorbound_geqC}, we know $C(u)-\delta(L_C+R^{DP}) \leq B(\B^{DP},u),$ which motivates us to define $\B^e$ as
        $$\B^e=\bigcap_{ u \in S^{N-1} } \Pi^-\left( u,B(\B^{DP},u) + \delta(L_C+R^{DP}) \right). $$
        By \Cref{lem:contour_extension} we know that $B(\B^e,u)= B(\B^{DP},u)+\delta(L_C+R^{DP})\geq C(u)$ for all $u\in S^{N-1}$. This yields firstly that
        $$ \int_{S^{N-1}}B(\B^e,u)d\sigma(u)
          =\int_{S^{N-1}}B(\B^{DP},u)d\sigma(u)+\delta(L_C+R^{DP}).
        $$
        Secondly, it implies that $\B^e$ is a feasible solution for the continuous problem, and therefore
        \begin{align*}
            V^{CP} 
            &\leq \int_{S^{N-1}}B(\B^e,u)d\sigma(u)\\
            &= \int_{S^{N-1}}B(\B^{DP},u)d\sigma(u)+\delta(L_C+R^{DP})\\
            &\leq \sum_{i=1}^m w_i B^*_i+2\epsilon R^{DP}+\delta(L_C+R^{DP})\\
            &= V^{DP} +2\epsilon R^{DP}+\delta(L_C+R^{DP}).
        \end{align*}
        Combining with \Cref{lem:Bmax<CmaxK}, this implies $V^{CP}-V^{DP} \le 12\sqrt N \|C\|_\infty(2\epsilon+\delta)+\delta L_C$.
        
        As for the other direction, we consider some optimal solution of the continuous problem $\B^{CP}$ and denote $R^{CP}=\max_{p\in\B^{CP}}\|p\|$. We know that for all $i=1,2,\dots,m$ there is some $p_i^{CP}\in\B^{CP}$ such that $\langle p_i^{CP},u_i\rangle \geq C(u_i)$. This means that $\left( \left(p^{CP}_i\right)_{i=1}^m , \, \left(B(\B^{CP},u_i)\right)_{i=1}^m \right)$ is a feasible solution of the linear programming problem. As a consequence, we have
        \begin{align*}
            V^{DP}
            &\leq \sum_{i=1}^m w_iB(\B^{CP},u_i)\\
            &\leq \int_{S^{N-1}}B(\B^{CP},u)d\sigma(u) + 2\epsilon R^{CP}\\
            &= V^{CP} + 2\epsilon R^{CP},
        \end{align*}
        which concludes the proof by recalling that \Cref{cor:Bmax<CmaxK_cont} says $R^{CP}\leq 4\sqrt N \|C\|_\infty.$
    \end{proof}
\end{thm}

This means that we can arbitrarily well approximate $V^{CP}$ by considering the linear programming problem \eqref{lpprob}. As a consequence of this result, we can also guarantee that the $(S^{N-1},C)$-valid contours constructed using \Cref{lem:contour_extension} can be made universally near-optimal in terms of the original continuous problem \eqref{cpprob}, by means of the following result.

\begin{thm}\label{thm:DPepsinCPeps_discrete}
  Fix some $\epsilon$-accurate quadrature $\{(u_i,w_i)\}_{i=1}^m$ and assume that $\SSS=\{u_i\}_{i=1}^m$ with $\delta=\disp(\SSS)$ satisfies $\epsilon,\delta \le \frac{1}{10\sqrt N}$. Let $e =\delta \big(L_C + 12\sqrt N(\|C\|_\infty+\gamma) \big)$ and define $\B^e =\bigcap_{ u \in \SSS } \Pi^-\left( u,B(\B,u) + e) \right)$ for any $\B\in\DD^\gamma$, $\gamma \ge 0$.
  
  We then have 
  $$ \B^e \in \CC^{\gamma+\beta}, $$
  where {$\beta = 4e + 32\sqrt N(\|C\|_\infty+\gamma)\, \epsilon$}.

    \begin{proof}
        We know from \Cref{lemm:errorbound_geqC,lem:contour_extension,lem:Bmax<CmaxK} that $\B^e$ satisfies 
        $$
        B(\B^e,u) 
        \geq B(\B,u) + e\\
        \geq  \big( C(u)-e \big) +e\\
        =C(u),
        $$
        for all $u\in S^{N-1}$, which implies $\B^e\in\CC^\infty$. Furthermore, we have 
        $$\sum_{i=1}^m w_iB(\B^e,u_i)
        =\sum_{i=1}^m w_i\big(B(\B,u_i)+e\big)
        \leq V^{DP}+\gamma+(1+\epsilon)e,$$
        which yields $ \B^e\in\DD^{\gamma+(1+\epsilon)e} $. This, along with the definition of $\epsilon$-accurate quadratures, \Cref{prop:B_lipz} and \Cref{lem:Bmax<CmaxK}, implies
        $$\left|\sum_{i=1}^m w_iB(\B^e,u_i)-\int_{S^{N-1}}  B(\B^e,u) d\sigma(u)\right|\leq 24\sqrt N\Big(\|C\|_\infty+\gamma+(1+\epsilon)e\Big) \epsilon.$$
        For ease of notation we denote $\alpha=24\sqrt N\Big(\|C\|_\infty+\gamma+(1+\epsilon)e\Big) \epsilon$. Combining these facts with \Cref{thm:errorbound_optval} then yields
        \begin{align*}
            \int_{S^{N-1}}B(\B^e,u)d\sigma(u)
            &\leq  \sum_{i=1}^m w_iB(\B^e,u_i) + \alpha\\
            &\leq V^{DP} +\gamma +(1+\epsilon)e+ \alpha \\
            &\leq V^{CP}+8\sqrt N\|C\|_\infty\epsilon + \gamma+(1+\epsilon)e+
            \alpha\\
            &\le V^{CP}+\beta+\gamma,
        \end{align*}
        which implies $\B^e \in \CC^{\gamma+\beta}$.
    \end{proof}
\end{thm}

%\green{It is worth remarking that while this result provides a way to guarantee $(S^{N-1},C)$-valid contours, it is still quite conservative. As we will see in \Cref{exam:BstarinBprime}, the construction of $\B'$ from \Cref{cor:Bprime} is sometimes valid by virtue of being an upper limit on sets corresponding to solutions of our discrete problem \eqref{lpprob}. }

\section{Numerical examples}\label{sec:examples}

\begin{figure}
  \includegraphics[width=0.45\linewidth]{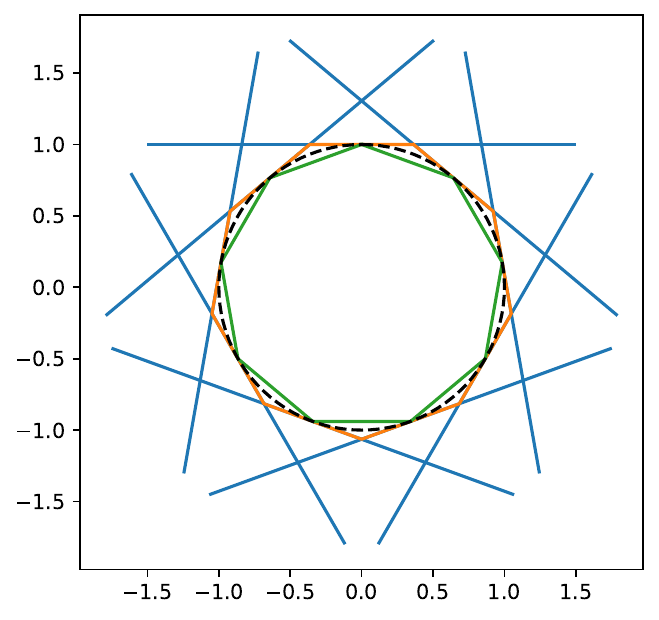}
  \vspace{-3mm}
  \caption{\label{fig:circle}We discretise the true outreach requirement (black circle) into only nine directions (blue). This highlights the difference between $\B^*$ from \Cref{prop:Bstar} (green) and $\B'$ from \Cref{cor:Bprime} (orange). $\B^*$ always lies inside of $\B'$.}
\end{figure}

\begin{exam}[Difference between $\B^*$ and $\B'$]\label{exam:BstarinBprime}
    \Cref{fig:circle} highlights the difference between the convex hull $\B^*$ from \Cref{prop:Bstar}, and the intersection of half-spaces $\B'$ from \Cref{cor:Bprime}. We use the unit circle outreach requirement $C = 1$ and sample it in nine evenly spaced directions. We see that the half-spaces defining $\B'$ each tangent the unit circle from the outside. On the other hand, the corners defining $\B^*$ lie on the circle, giving a convex hull inside the circle.

    This means that using $\B'$ will be a more conservative estimate, and therefore a safer choice. In this particular case we even have $\B'\in\CC^\infty$, making it a $(C,S^{N-1})$-valid contour without the need for inflation as per \Cref{thm:DPepsinCPeps_discrete}.
\end{exam}

\begin{figure}
  \includegraphics[width=0.485\textwidth]{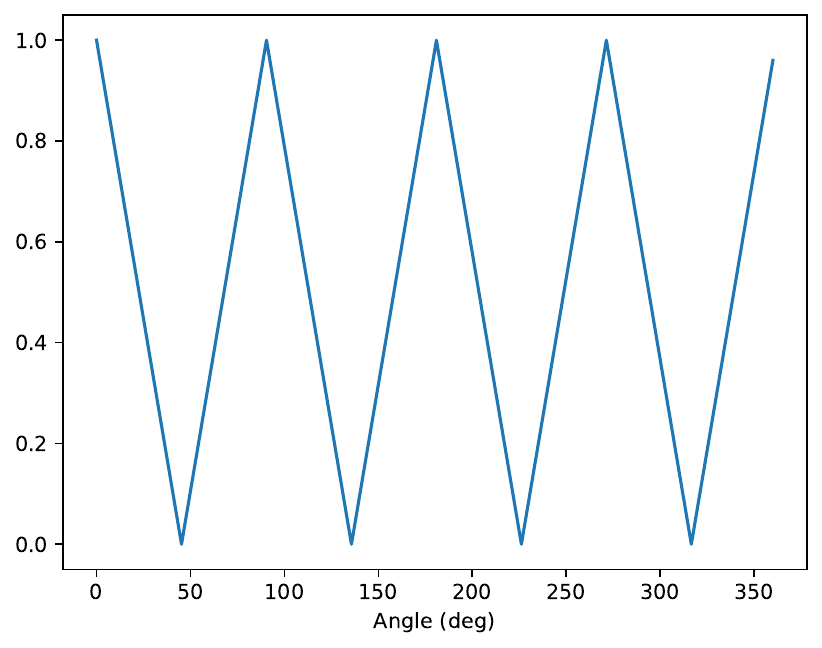} 
  \includegraphics[width=0.415\textwidth]{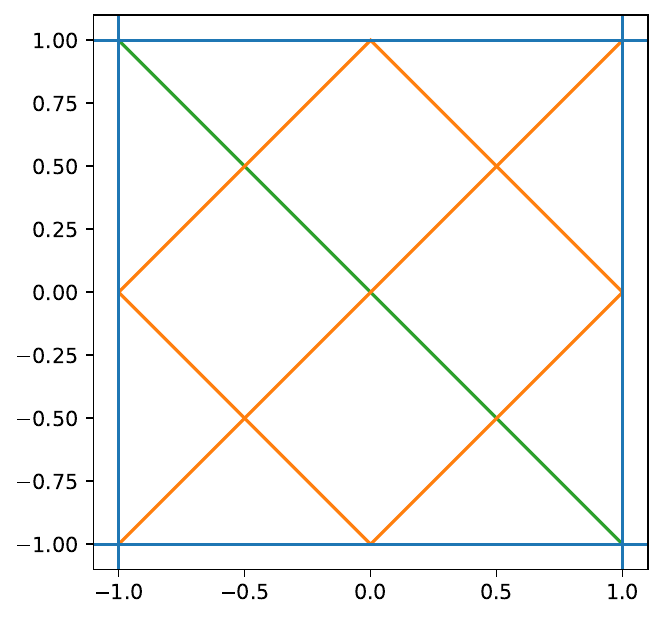} 
  \vspace{-3mm}
  \caption{\label{fig:square}Left: Periodic constraint function $C$ constructed such that \Cref{cpprob} has several shapes with minimal perimeter. Right: The outreach requirement boils down to requiring the shape to touch the four sides of a square (blue). A specific numerical implementation of our method \eqref{lpprob_2d1} output the green shape. We highlight two other optimal shapes (orange).}
\end{figure}

\begin{exam}[Multiple optimal shapes]\label{exam:square}
    We construct a Lipschitz-continuous outreach requirement (\Cref{fig:square}, left) which requires the shape to reach outreach 1 in each of the four cardinal directions. For this requirement, there are infinitely many shapes with the optimal perimeter. There are two distinct solution minimising the area: the two diagonal line segments with zero area. However, interpolating the two diagonal extremes are infinitely many 45 degree tilted rectangles, also with optimal perimeter.
\end{exam}

\begin{figure}
  \includegraphics[width=0.45\linewidth]{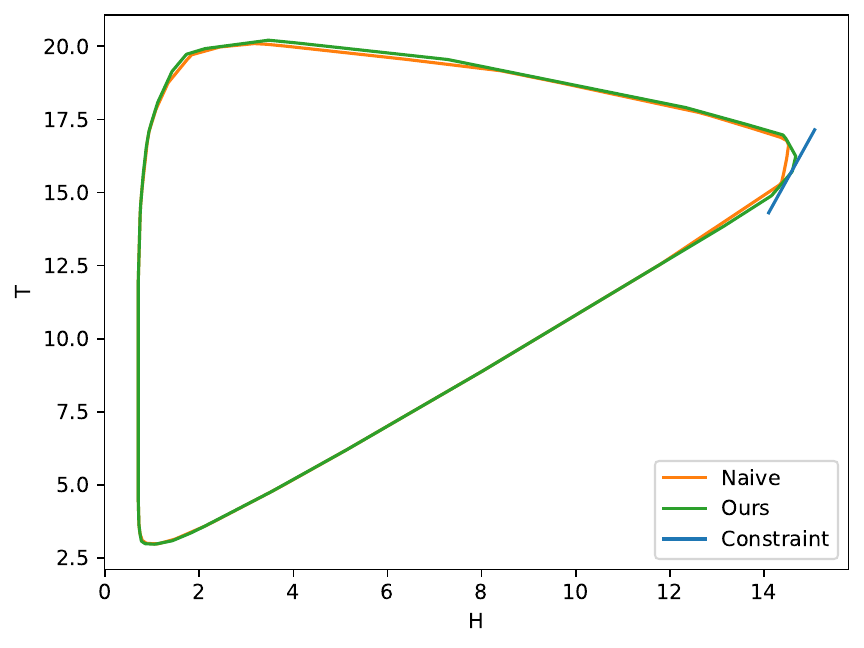}
  \includegraphics[width=0.45\linewidth]{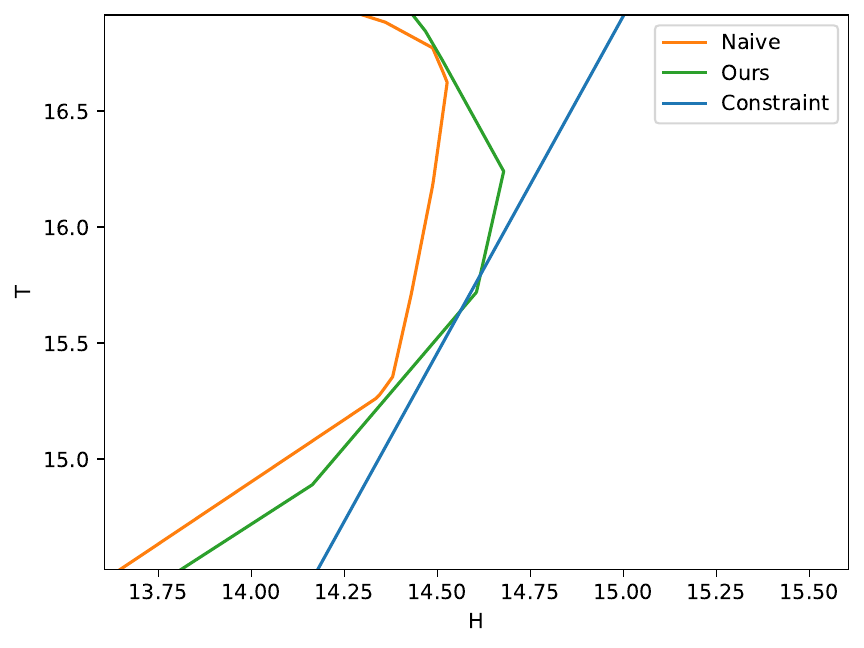}
  \vspace{-3mm}
  \caption{\label{fig:arne2d}The naive method fails to output a convex shape satisfying the requirements, when the outreach function $C$ contains noise. {We highlight a violated outreach constraint in blue. The plots are the same, with the right one being zoomed in.} }
\end{figure}

\begin{exam}[Comparing with naive method]\label{exam:2d}
    In \cite{altcontour}, they compute $C$ under the assumption that the environmental loads, $V$, are modeled as a sequence of i.i.d.\ random variables $\{W_n\}_{n=0}^\infty$. Specifically, they assume that $V_t=W_{\lfloor t/\Delta t\rfloor}$ for a time increment $\Delta t\in \R_+$, where $\lfloor \cdot  \rfloor$ denotes the floor function. They then define $C(u)$ by the upper $p_e$-quantile of $\langle W,u \rangle$, for a target exceedance probability $p_e$. This will guarantee that the mean time to failure for any convex failure set not intersecting with a $(C,S^{N-1})$-valid contour is at least $\Delta t/p_e$. We select $p_e = {1}/{29200}$ and $\Delta t=3$ hours, implying a 10 year lower bound on the mean time to failure.

    The suggested method presented in the aforementioned article \cite{altcontour}, as well as the recommended practises of DNV \cite{veritas2000environmental}, is the following. Find a model for $W$, simulate a number of samples, and use the empirical quantiles of $\langle W,u \rangle$ as an estimate of $C(u)$ for a finite selection of directions. We choose $3\times 10^5$ samples in $360$ uniformly spaced directions.
    
    In \cite{altcontour}, $W$ was modelled as $W=(H,T)$. Here $H$ is a 3-parameter Weibull-distributed random variable in $\R$ representing significant wave height. $H$ has scale $2.259$, shape $1.285, $ and location $0.701$. Similarly, $T$ represents the zero-upcrossing wave period and is assumed to follow a conditional log-normal distribution, i.e.\ $\ln(T)$ is normally distributed with with conditional mean $1.069 +0.898H^{0.243},$ and conditional standard deviation $   0.025 + 0.263e^{-0.148H}.$

    It is further suggested to construct $\B$ by
    \begin{equation}\label{eq:naiveconstruction}
        \B=\bigcap_{u\in S^{N-1}}\Pi^-(u,C(u)).
    \end{equation}
    We will henceforth refer to \eqref{eq:naiveconstruction} as the \textit{naive method}. 
    
    The estimated $C$ is slightly noisy. As pointed out in e.g.\ \cite{convcont,voronoi}, this noise may causes the naive method to fail to satisfy the outreach requirements. For an explanation on why this occurs, we refer to either \cite{convcont,voronoi} or the discussion around \Cref{fig:loopnoloop}.
    
    As seen in \Cref{fig:arne2d}, the naive method outputs an improper contour. Using the naive method with the $C$ from \Cref{exam:square}, would cause it to output the single point $(0,0)$, which does not satisfy the outreach requirements.
\end{exam}

\begin{figure}
  \includegraphics[width=0.45\linewidth]{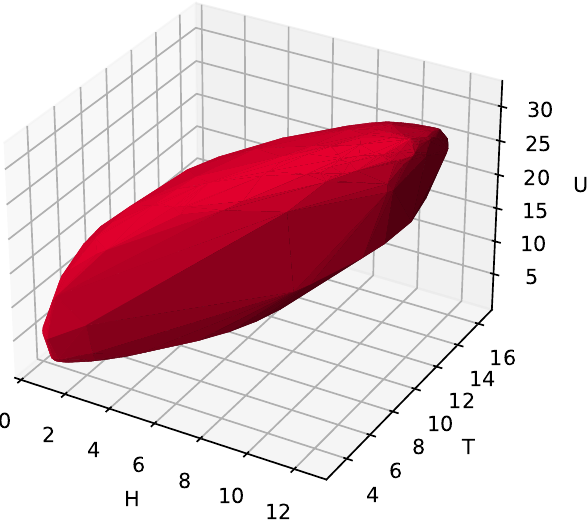}
  \vspace{-3mm}
  \caption{\label{fig:arne3d}Shape found by \Cref{lpprob} for the 3d example described in \Cref{exam:3d}.}
\end{figure}

\begin{exam}[Three dimensions]\label{exam:3d}
    To demonstrate our method in three dimensions, we use an example from \cite{vanem20193}. The construction of $C$ is similar to \Cref{exam:2d}. The sequence $W = (H,T,U)$ distributed as follows. $H$ follows a 3-parameter Weibull distribution with scale $1.798$, shape $1.214$ and location $0.856$. Given $H$, $\ln(T)$ is normally distributed with mean $-1.010+2.847 H^{0.075}$ and standard deviation $0.161+0.146 e^{-0.683 H}$. Finally, given $H$, $U$ follows a 2-parameter Weibull distribution with scale $2.58+0.12 H^{1.6}$ and shape $4.6+2.05 H$. We use the empirical quantiles with $10^6$ samples of $\langle W,u \rangle$ to estimate $C(u)$, and select $p_e = \frac{1}{29200}$.
    
    For discretization, we sample $C$ according to the cubed hypersphere quadrature with 10 subdivisions (\Cref{cubed_hypersphere_accurate}). Then we solve \Cref{lpprob} and compute the convex hull $\B^*$ defined in \Cref{prop:Bstar}. \Cref{fig:arne3d} visualizes the resulting contour.
\end{exam}

\begin{figure}
  \includegraphics[width=0.45\textwidth]{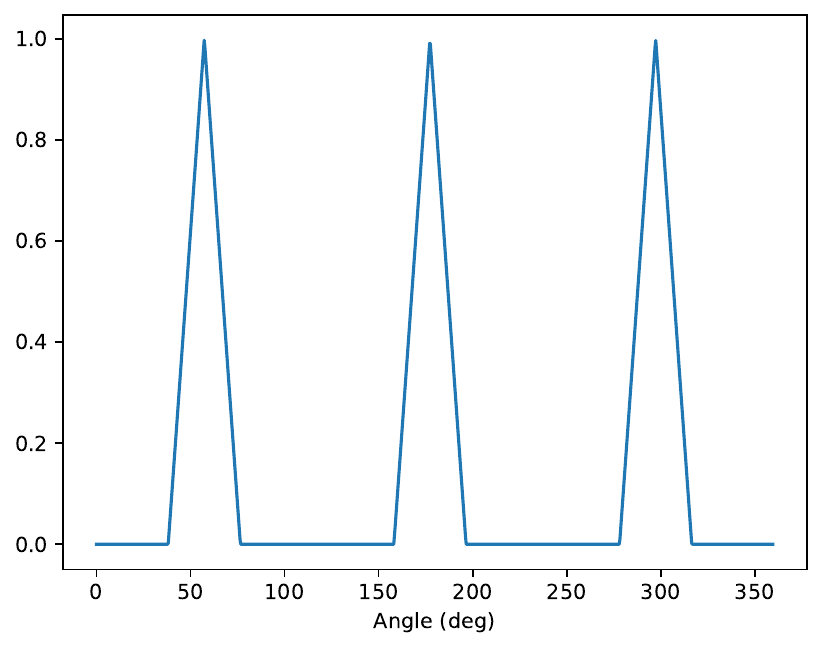} 
  \includegraphics[width=0.455\textwidth]{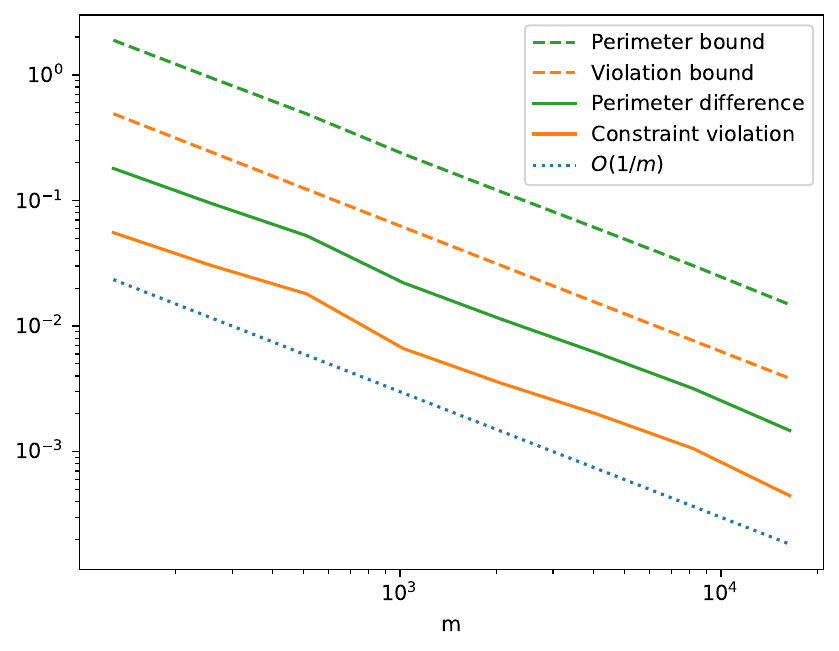} 
  \vspace{-3mm}
  \caption{\label{fig:lip_tri}Left: A 3-Lipschitz continuous constraint function $C$.
  Right: The empirical convergence rate of our method matches the rate $O(1/m)$ given by our theory (\Cref{thm:DPepsinCPeps_discrete}).}
\end{figure}
\begin{exam}[Empirical convergence rate]\label{exam:lip_tri}
    \Cref{fig:lip_tri} demonstrates the empirical convergence rate of our method on a 3-Lipschitz function consisting of three spikes. The optimal shape, i.e.\ the sole element of $\CC^0$, is an equilateral triangle. Since $C$ is only sampled in a finite number of directions, and these directions rarely exactly match the spikes of $C$, the samples underestimate $C$. This leads the direct (not inflated) discrete solution $\B'$ to violate the constraint \eqref{cpconstr} in unseen directions. We plot the largest such constraint violation, i.e., $\sup_{u\in S^1} C(u)-B(\B',u)$, and see that it is indeed bounded as per \Cref{thm:DPepsinCPeps_discrete}, which guarantees convergence of order $O(1/m)$. The perimeter difference between the solution $\B'$ found by our method and the optimal solution $\mathcal{B}^*$, i.e., \[2\pi\left|\int_{S^{1}} B(\B',u) d\sigma(u) - \int_{S^{1}} B(\B^*,u) d\sigma(u)\right|,\] also converges like $O(1/m)$ as predicted by the bound from \Cref{thm:DPepsinCPeps_discrete}.
\end{exam}

\begin{figure}
  \includegraphics[width=0.450\textwidth]{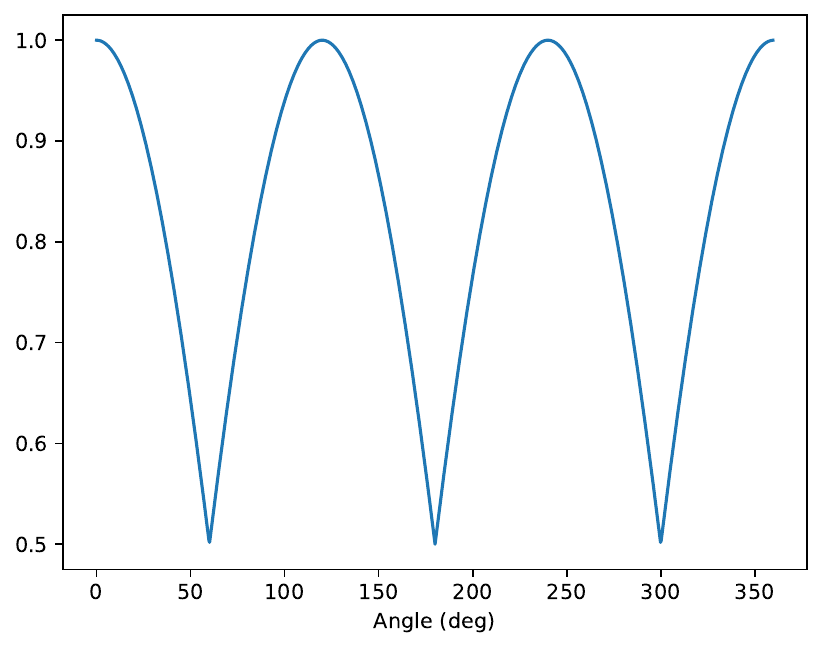} 
  \includegraphics[width=0.455\textwidth]{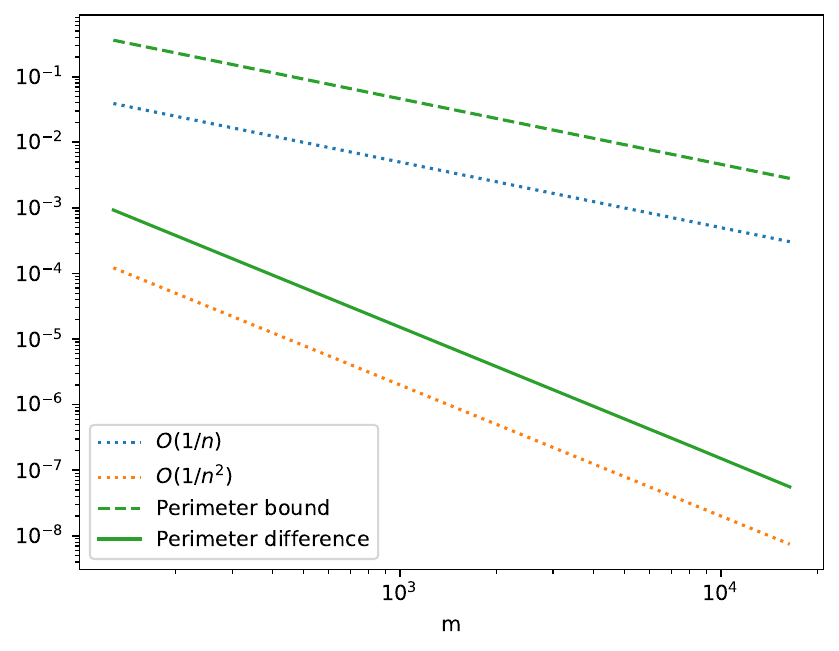} 
  \vspace{-3mm}
  \caption{\label{fig:nice_tri}Left: A  outreach function  corresponding to the outreach function of an equilateral triangle. Right: The empirical convergence rate for this case seems to be $O(1/m^2)$.}
\end{figure}
\begin{exam}[Empirical convergence is often $O(1/m^2)$]\label{exam:nice_tri}
    While the convergence rate $O(1/m)$ can not be improved in general, as seen in \Cref{exam:lip_tri}, in practice, we often see the faster rate $O(1/m^2)$. For example, we may choose $C$ as the outreach function of an equilateral triangle, i.e.,
    \[C(u) = \max\{\cos(\theta),\ \cos(\theta-2\pi/3),\ \cos(\theta-4\pi/3)\},\]
    for $u=\left(\cos(\theta),\sin(\theta)\right)$. In this case, the solution $\mathcal{B}'$ found by our method always satisfies the continuous constraints \eqref{cpconstr} perfectly, and we empirically find a convergence rate of $O(1/m^2)$. Note that while $C$ is $\sqrt{3}/2$-Lipschitz continuous, it is not everywhere differentiable, and the optimal shape (equilateral triangle) has sharp corners.
\end{exam}

\section{Connection with Hausdorff Topology}\label{sec:hausdorff}

While the previous section guarantees the construction of arbitrarily near-optimal $(S^{N-1},C)$-valid contours, there is still the question of whether all elements of $\CC^0$ can be approximated in this way. In order to examine this question, we will need the concept of the Hausdorff metric. The resulting topology turns out to be a natural framework for examining convergence of the discrete approximations of our continuous problem, but it will also allow us to properly prove our earlier claim that the continuous problem indeed has optimal solutions.

\subsection{Basic Concepts and Definitions}

The main tool we will use here is the Hausdorff metric, defined as follows.

\begin{defi}
    Let $(X,d)$ be a metric space, and let $F(X)$ denote the collection of all non-empty, compact subsets of $X$. For any $x\in X$ and $A,\,B\in F(X)$, we can define $d(x,B)=d(B,x)=\min_{b\in\B}d(x,b)$ and
    $$d_\HH(A,B)=\max\left( \max_{a\in A}d(a,B),\,\max_{b\in B}d(A,b) \right).$$
    The set-function $d_\HH$ is referred to as the Hausdorff distance.
\end{defi}

We have the following basic properties of $d_\HH$, for a proof of these properties we refer to \cite{Hausdorffproperties}.

\begin{prop}\label{prop:haus_properties}
  The space $(F(X),d_\HH)$ is a metric space. Furthermore, if $(X,d)$ is a complete and compact metric space, then $(F(X),d_\HH)$ is complete and compact as well. Lastly, if $X$ is a Banach space and $\{A_n\}_{n=1}^\infty\subset F(X)$ is a sequence of convex sets converging to some set $A\in F(X)$, then $A$ is also convex.
\end{prop}

In what follows, we will consider $(\R^N,d)$ where $d$ is the canonical Euclidean metric on $\R^N$. This allows us to define the Hausdorff metric on $F(\R^N)$, but also allows us to discuss $F(F(\R^N))$, i.e.\ compact collections (in $d_\HH$) of compact subsets of $\R^N$ and the associated metric, $d_{\HH_\HH}$. With these definitions we can discuss $d_{\HH_\HH}(\CC^0,\DD^\gamma)$ which quantifies how well our discrete problem can approximate the \textit{entirety} of $\CC^0$.

\subsection{Existence of Solutions to the Continuous Problem}
Our first point of order is our previous claim of existence of solutions to \Cref{cpprob}.

\begin{thm}\label{thm:CP_sol_exist}
    Define the objective function $\phi \colon \CC^\infty\mapsto [V^{CP},\infty)$ by
        $$\phi(\B)=\int_{S^{N-1}} B(\B,u) d\sigma(u).$$
    We then have that $\phi$ is Lipschitz continuous in $(F(\R^N),d_\HH)$ with Lipschitz constant $1$. Furthermore, for any $\gamma\geq 0$, $\CC^\gamma\subset F(\R^N)$ is non-empty and compact in the resulting Hausdorff topology. As a specific consequence of this, the continuous problem \eqref{cpprob} has at least one optimal solution.

    \begin{proof}
        In what follows, we will need the following relation. If $\B_1,\,\B_2 \subset \R^N$ are compact then $|B(\B_1,u)-B(\B_2,u)|\leq d_\HH(\B_1,\B_2)$ for all $u\in S^{N-1}$. To see this, we pick $p_1\in\B_1$ such that $ \langle p_1 , u \rangle=B(\B_1,u) $. From the definition of the Hausdorff distance, there must be some $p_2\in\B_2$ such that $\|p_1-p_2\| \leq d_\HH(\B_1,\B_2)$ which yields
        \begin{align*}
            B(\B_1,u)
            &=\langle p_1,u \rangle \\
            &= \langle p_1-p_2,u \rangle + \langle p_2,u \rangle \\
            &\leq d_\HH(\B_1,\B_2) + B(\B_2,u).
        \end{align*}
        This proves that $B(\B_1,u)-B(\B_2,u)\leq d_\HH(\B_1,\B_2)$, and an identical argument with $B_1$ and $B_2$ interchanged gives $B(\B_2,u)-B(\B_1,u)\leq d_\HH(\B_1,\B_2)$. This implies that for any $u\in S^{N-1}$, the function $\B\mapsto B(\B,u)$ is Lipschitz continuous with Lipschitz constant $1$.

        Next, we fix some $\gamma\geq0$ and aim to prove compactness of $\CC^\gamma$, leaving non-emptiness for later. Note that $\CC^\gamma$ is either non-empty or trivially compact, as such we can, without loss of generality, assume that $\CC^\gamma$ is non-empty. Furthermore, by \Cref{cor:Bmax<CmaxK_cont}, we note that for every $\B\in\CC^\gamma$ we have $\B\subset K_\gamma$ where 
        $$K_\gamma = \{x\in\R^N:\|x\|\leq 4\sqrt N(\|C\|_\infty+\gamma)\}.$$
        If we denote by $F(K_\gamma)$ the collection of all compact subsets of $K_\gamma$, we have that $\CC^\gamma\subset F(K_\gamma)$. Furthermore, by \Cref{prop:haus_properties}, we have that $F(K_\gamma)\subset F(\R^N)$ equipped with the Hausdorff metric is a complete and compact metric space.

        We next aim to prove that $\CC^\gamma$ is a closed (and therefore compact) subset of $F(K_\gamma)$ in the Hausdorff topology. To see this, consider any convergent sequence $\{B_n\}_{n=1}^\infty\subset \CC^\gamma \subset F(K_\gamma)$ with $\B_n\to \B$. By \Cref{prop:haus_properties} we have that $\B$ is convex, and by completeness of $F(K_\gamma)$ we have $\B$ compact. Furthermore, by continuity of $\B\mapsto B(\B,u)$, we have for any $ u\in S^{N-1} $ that
        $$B(\B,u)=\lim_{n\to\infty}B(\B_n,u)\geq C(u),$$
        implying $\B\in\CC^\infty$. Lastly, by $B(\B_n,u)\leq \max_{p\in \B_n}\|p\|\leq 4\sqrt N (\|C\|_\infty+\gamma)$ for all $u\in S^{N-1}$, we can apply the dominated convergence theorem to get
        \begin{align*}
            \int_{S^{N-1}} B(\B,u) d\sigma(u)
            &=\int_{S^{N-1}} \lim_{n\to\infty}B(\B_n,u) d\sigma(u)\\
            &=\lim_{n\to\infty}\int_{S^{N-1}} B(\B_n,u) d\sigma(u)\\
            &\leq V^{CP}+\gamma,
        \end{align*}
        which implies $\B\in\CC^\gamma$. As a consequence, for any $\gamma\geq0$, $\CC^\gamma$ is a closed subset of the compact space $F(K_\gamma)$, and therefore itself compact. 

        As for non-emptiness of $\CC^\gamma$ we first remark that $\CC^\gamma$ is non-empty for any $\gamma > 0$. To see this, note by the definition of $V^{CP}$ that there exists either an optimal solution in $\CC^0$ with mean width $V^{CP}$ or a sequence of near-optimal solutions with mean width converging to $V^{CP}$. Either way, this implies that $\CC^\gamma$ is non-empty for any $\gamma>0$.
        
        Next, assume $\gamma >0$ and note that for any $\B_1,\,\B_2\in\CC^\gamma$ we have
        \begin{align*}
            |\phi(\B_1)-\phi(\B_2)|
            &=\left|\int_{S^{N-1}} (B(\B_1,u)-B(\B_2,u)) d\sigma(u)\right|\\
            &\leq \int_{S^{N-1}} \left|B(\B_1,u)-B(\B_2,u)\right| d\sigma(u)\\
            &\leq \int_{S^{N-1}} d_\HH(\B_1,\B_2) d\sigma(u)\\
            &=d_\HH(\B_1,\B_2).
        \end{align*}
        This means that $\phi$ is Lipschitz, so it must attain a minimal value on $\CC^\gamma$ for any $\gamma > 0$. This minimum, by definition, will also be a minimum on $\CC^\infty$. Consequently, this minimal point, $\B_\text{min}$, must satisfy $\phi(\B_\text{min})=V^{CP}$, yielding $\CC^0$ non-empty.
    \end{proof}
\end{thm}

\subsection{Convergence in Hausdorff Metric}
With existence of solutions settled, along with some useful properties of $\phi$ and $\CC^\gamma$, we can move on to our main goal of this section: To prove that we can arbitrarily well approximate the entirety of $\CC^0$ by $\DD^0$ or $\DD^\gamma$. To do so we will consider 
\begin{equation}\label{eq:dCD}
  d_{\HH_\HH}(\CC^0,\DD^\gamma)=\max\left(
\max_{\B'\in\CC^0}\min_{\B\in\DD^\gamma}d_\HH(\B,\B'),
\max_{\B'\in\DD^\gamma}\min_{\B\in\CC^0}d_\HH(\B,\B')
\right).  
\end{equation}
The primary goal of this section is to show that for a sequence of quadratures, $\{(\SSS_n,W_n)\}_{i=1}^\infty$, we have $\{\gamma_n\}_{i=1}^\infty$, with $\gamma_n\to0$ such that  $d_{\HH_\HH}(\CC^0,\DD^{\gamma_n}(\SSS_n,W_n))\to0$.  This would imply that for any given $\B'\in\CC^0$ there is some set $\B\in \DD^{\gamma_n}(\SSS_n,W_n) $ that approximates it, meaning that \textit{all} optimal solutions can be approximated by our discrete solutions. Conversely, for any $\B\in \DD^{\gamma_n}(\SSS_n,W_n) $, we know it will be close to some $\B\in\CC^0$ in Hausdorff distance, implying that our discrete solutions will get closer and closer to our continuous solutions.

It turns out that if we are only interested in guaranteeing that our discrete solutions are close to a continuous solution we can drop the inclusion of $\gamma_n$. In particular we will have
$\max_{\B'\in\DD^0(\SSS_n,W_n)}\min_{\B\in\CC^0}d_\HH(\B,\B') \to 0.$
Furthermore, if the continuous problem \eqref{cpprob} has a unique solution, then $d_{\HH_\HH}(\CC^0,\DD^{0}(\SSS_n,W_n))\to0$.

However, in order to even consider this distance we will first need to guarantee that $\DD^\gamma$ is a compact set in the Hausdorff metric.
\begin{cor}\label{cor:Dcompact}
    Assume that $ \{(u_i,w_i)\}_{i=1}^m\subset \SSS\times [0,\infty) $ is an $\epsilon$-accurate quadrature with $\delta \coloneqq \disp(\SSS)$ satisfying $\epsilon,\delta \le \frac{1}{10\sqrt N}$. We then have that $\DD^\gamma$ is compact in $d_\HH$.
    \begin{proof}
        We first note that \Cref{lem:Bmax<CmaxK} yields $\DD^\gamma\subset F(K_\gamma')$ for
        $$K_\gamma' = \{x\in\R^N:\|x\|\leq 12\sqrt N(\|C\|_\infty+\gamma)\}.$$
        By defining
        $$\psi(\B)=\sum_{i=1}^m w_i B(\B,u_i),$$
        we can then repeat the arguments of \Cref{thm:CP_sol_exist} with $K_\gamma',\psi$ replacing $K_\gamma,\phi$, which proves the desired result.
    \end{proof}
\end{cor}
With this established we know that \eqref{eq:dCD} is well defined, and what remains is finding ways to bound it. We remark that $\max_{\B'\in\CC^0}\min_{\B\in\DD^\gamma}d_\HH(\B,\B')$ in \eqref{eq:dCD} can be bounded by considering the triangular inequality
\begin{equation}\label{eq:triangle}
  d_\HH(\B,\B')\leq d_\HH(\B,\B^e) + d_\HH(\B^e,\B'),  
\end{equation}
where $\B^e=\bigcap_{ v \in S^{N-1} } \Pi^-\left( v,B(\B,v) + e \right)$ for some appropriate $e$.

The first term of the right side is easily dealt with by extending the results of \Cref{lem:contour_extension} to the setting of $d_{\HH}$.

\begin{lemm}\label{lem:contour_extension_V2}
    Let  $ \B $ be a convex and compact set and define
    $$\B^e=\bigcap_{ v \in S^{N-1} } \Pi^-\left( v,B(\B,v) + e \right).$$
    We then have that $d_\HH(\B,\B^e)=e$.
    \begin{proof}
        Consider the alternative construction 
        $$\widehat\B=\{ x\in\R^N: d(x,\B)\leq e \}. $$
        We immediately see that $\widehat\B$ is a compact convex set satisfying $d_\HH(\B,\widehat \B)=e$, so all we need is to show $\B^e=\widehat\B$.

        To prove this, we note that for all $x\in\widehat\B$ we can decompose $x=b+c$ where $b\in\B$ and $\|c\|\leq e$. Conversely, if $b\in\B$ and $\|c\|\leq e$ we have $b+c\in\widehat\B$. Using this we get, for any $u\in S^{N-1}$, that
        \begin{align*}
            B(\widehat\B,u)
            &=\max \{ \langle x,u\rangle: x\in\widehat\B\}\\
            &=\max \{ \langle b,u\rangle +\langle c,u\rangle: b\in\B,\, \|c\|\leq e\}\\
            &=B(\B,u)+e.
        \end{align*}
        We may also recall from \cref{lem:contour_extension} that $B(\B^e,u)=B(\B,u)+e=B(\widehat\B,u)$, for all $u\in S^{N-1} $. This implies that $\B^e=\widehat\B$ by \Cref{prop:support_rep}, which guarantees uniqueness of representation by $B$.

    \end{proof}
\end{lemm}

The second term of the right side of \eqref{eq:triangle} requires two steps to control. By \Cref{thm:errorbound_optval}, we can bound the mean width of $\B^e$ by the following result. This result is almost identical to \Cref{thm:DPepsinCPeps_discrete}. In that result, however, we consider the set $\B^e$ to be an inflation of $\B$ only in directions $u\in\SSS$. In the following result we consider an inflation of $\B$ in all directions $u\in S^{N-1}$.

\begin{prop}\label{prop:DPepsinCPeps_continuous}
  Fix some $\epsilon$-accurate quadrature $\{(u_i,w_i)\}_{i=1}^m$ and assume that $\SSS=\{u_i\}_{i=1}^m$ with $\delta=\disp(\SSS)$ satisfies $\epsilon,\delta \le \frac{1}{10\sqrt N}$. Let $e =\delta \big(L_C + 12\sqrt N(\|C\|_\infty+\gamma) \big)$ and define $\B^e =\bigcap_{ u \in S^{N-1} } \Pi^-\left( u,B(\B,u) + e) \right)$ for some $\B\in\DD^\infty$.
 Note here that we consider $u \in S^{N-1}$ and not $u \in \SSS$ as in \Cref{thm:DPepsinCPeps_discrete}.
 
  Then, for any $\B\in\DD^\gamma$, $\gamma \ge 0$, we have that 
  $$ \B^e \in \CC^{\gamma+\beta}, $$
  where $\beta = e + 32\sqrt N(\|C\|_\infty+\gamma)\, \epsilon$.

    \begin{proof}
        %We know from \cref{lemm:errorbound_geqC,lem:contour_extension,lem:Bmax<CmaxK} 
        We know from \Cref{lemm:errorbound_geqC,lem:contour_extension,lem:Bmax<CmaxK} that $\B^e$ satisfies 
        \begin{align*}
            B(\B^e,u) 
            &= B(\B,u) + \delta(L_C+12\sqrt N(\|C\|_\infty+\gamma))\\
            &\geq  \Big( C(u)-e \Big) +\delta(L_C+12\sqrt N(\|C\|_\infty+\gamma))\\
            &=C(u)
        \end{align*}
        for all $u\in S^{N-1}$, which implies $\B^e\in\CC^\infty$. Furthermore, we have by the definition of $\epsilon$-accurate quadratures, \Cref{prop:B_lipz} and \Cref{lem:Bmax<CmaxK} that
        $$\left|\sum_{i=1}^m w_iB(\B,u_i)-\int_{S^{N-1}}  B(\B,u) d\sigma(u)\right|\leq 24\sqrt N(\|C\|_\infty+\gamma) \epsilon.$$
        
        Combining these facts with \Cref{thm:errorbound_optval} then yields
        \begin{align*}
            \int_{S^{N-1}}B(\B^e,u)d\sigma(u)
            &= \int_{S^{N-1}}B(\B,u)d\sigma(u) + e\\
            &\leq  \sum_{i=1}^m w_iB(\B,u_i) + 24\sqrt N(\|C\|_\infty+\gamma) \epsilon + e\\
            &\leq V^{DP}+\gamma + 24\sqrt N(\|C\|_\infty+\gamma) \epsilon + e\\
            &\leq V^{CP}+8\sqrt N\|C\|_\infty\epsilon+\gamma+24\sqrt N(\|C\|_\infty+\gamma) \epsilon + e\\
            &\le V^{CP}+\gamma+\beta,
        \end{align*}
        which implies $\B^e \in \CC^{\gamma+\beta}$.
    \end{proof}
\end{prop}

With this result, for any $\B'\in\CC^0$ and $\B^e$ defined as in \Cref{prop:DPepsinCPeps_continuous}, we have $d_\HH(\B',\B^e)\leq d_{\HH_\HH}(\CC^\gamma,\CC^0)$. This can be controlled by the following.

\begin{prop}\label{prop:CcloseC}
    Consider the Hausdorff metric on the set $F(\CC^\infty)$, i.e.\ the set of all compact subsets of $\CC^\infty$ where $\CC^\infty$ itself is also equipped with a Hausdorff metric. We then have that 
    $$\lim_{\gamma \to 0}d_{\HH_\HH}(\CC^\gamma,\CC^0)=0.$$

    \begin{proof}
        We first recall from \Cref{thm:CP_sol_exist} that $\CC^\gamma$ is compact and that $\phi:\CC^\infty \to [V^{CP},\infty)$ defined by $\phi(\B)=\int_{S^{N-1}}B(\B,u)d\sigma(u)$ is a Lipschitz continuous function with Lipschitz constant $1$. Also, note that since $\CC^0\subset\CC^\gamma$, we have
        $$d_{\HH_\HH}(\CC^\gamma,\CC^0)=\max_{\B\in\CC^\gamma}\min_{\B'\in\CC^0}d_\HH(\B,\B').$$

        Next, pick some $\gamma,\beta>0$. Since $\CC^\gamma$ is compact, we can consider a finite open covering $\{\U_i\}_{i=1}^I$ such that $\diam(\U_i)\leq\beta$ for all $i=1,\dots,I$. Note that by Lipschitz continuity of $\phi$, we have for all $i=1,\dots,I$ that $\phi(\B)\leq \gamma+\beta$ for any $\B\in\U_i$. As a consequence, if we then define $\overline{\U_i}$ as the closure of $\U_i$ in the Hausdorff metric, we know for all $i=1,\dots,I$, that $\overline{\U_i}$ must be a closed subset of the compact set $\CC^{\gamma+\beta}$, which implies that $\overline{\U_i}$ is compact as well.

        This implies that $\phi$ must attain a minimum on $\overline{\U}_i$ which lets us define $\{\psi_i\}_{i=1}^I$ by $\psi_i=\min_{\B\in\U_i}\phi(\B)$. We then separate the sets that overlap with $\CC^0$ by defining $ \I=\{i=1,\dots,I:\psi_i=V^{CP}\} $ and $\J=\{i=1,\dots,I:\psi_i > V^{CP}\} $. If we then define $\psi_{min}=\min_{j\in \J}\psi_j>V^{CP}$, we can pick some $\alpha\in \big(0,\min(\gamma,\psi_{\min}-V^{CP})\big)$ which implies that $\phi(\B)\geq \psi_{min} > V^{CP}+\alpha $ for any $\B\in\cup_{j\in \J}\overline{\U}_j$. As a consequence, we have
        $$\bigcup_{j\in \J}\overline{\U}_j \cap \CC^{\alpha}=\emptyset.$$
        However, since $\{\U_i\}_{i=1}^I$ is a covering of $\CC^\gamma \supset \CC^{\alpha}$ we must have that
        $$\CC^{\alpha} \subseteq \bigcup_{i\in \I}\overline{\U}_i.$$

        This implies that for every $\B\in\CC^{\alpha}$ we have $\B\in \overline{\U}_{i'}$ for some $i'\in\I$. Since $\psi_{i'}=V^{CP}$, there is some $\B'\in \overline{\U}_{i'}$ such that $\phi(\B')=V^{CP}$ which implies $\B'\in \CC^0$. By $\diam(\overline{\U}_i)=\diam(\U_i)\leq\beta$ we then finally get that $d_\HH(\B,\B')\leq\beta$, but since $\B$ was arbitrary we also have $d_{\HH_\HH}(\CC^0,\CC^{\alpha})\leq\beta$.

        In summary, we see that for every $\beta>0$ there exists some $\alpha>0 $ such that $d_{\HH_\HH}(\CC^0,\CC^{\alpha})\leq\beta$,  which implies $\lim_{\gamma \to 0}d_{\HH_\HH}(\CC^0,\CC^\gamma)=0$ by monotonicity.

    \end{proof}
\end{prop}

    It turns out that the triangle inequality of \eqref{eq:triangle} is not sufficient to control \eqref{eq:dCD}. In particular, we also need to handle $\max_{\B'\in\DD^\gamma}\min_{\B\in\CC^0}d_\HH(\B,\B')$, which equals 0 if $\CC^0 \subset \DD^\gamma$. Fortunately, a $\gamma$ such that this is satisfied can be attained.

\begin{lemm}\label{prop:CPinDPeps}
  Fix some $\epsilon$-accurate quadrature $\{(u_i,w_i)\}_{i=1}^m$ and assume that $\SSS=\{u_i\}_{i=1}^m$ with $\delta=\disp(\SSS)$ satisfies $\epsilon,\delta \le \frac{1}{10\sqrt N}$.
    
    Then $ \CC^0 \subset \DD^\gamma $ for $\gamma=16\sqrt N\|C\|_\infty(2\epsilon+\delta)+L_C\delta$.

    \begin{proof}
        Take any $\B\in\CC^0$, we then have, for all $u\in\SSS$, that $B(\B,u)\geq C(u)$, which implies $\B\in\DD^\infty$. Furthermore, \Cref{prop:B_lipz} and \Cref{cor:Bmax<CmaxK_cont} give us
        $$\left|V^{CP}-\sum_{i=1}^m w_i B(\B,u_i) \right|\leq 8\sqrt N\|C\|_\infty\epsilon.$$
        
        Combining these facts with \Cref{thm:errorbound_optval} yields
        \begin{align*}
            \sum_{i=1}^m w_i B(\B,u_i)
            &\leq V^{CP}+8\sqrt N\|C\|_\infty\epsilon\\
            &\leq V^{DP}+12\sqrt N\|C\|_\infty(2\epsilon+\delta)+L_C\delta+8\sqrt N\|C\|_\infty\epsilon\\
            &\le V^{DP}+\gamma,
        \end{align*}
        which implies $\B \in \DD^{\gamma}$.
    \end{proof}
\end{lemm}

    With all our technical results, we are ready to prove the second main result of this section.
\begin{thm}\label{thm:dDnC_to0}
    Consider a sequence, $\{(\SSS_n,W_n)\}_{i=1}^\infty$, with $\SSS_n=(u_i^n)_{i=1}^{I_n}$ and $W_n=(w_i^n)_{i=1}^{I_n}$ such that $ \{(u_i^n,w_i^n)\}_{i=1}^{I_n} \subset \SSS_n \times W_n$ forms an $\epsilon_n$-accurate quadrature. Further assume that $\epsilon_n$ and $\delta_n=\disp(\SSS_n)$ both converge to 0 and satisfy $\epsilon_n,\delta_n \le \frac{1}{10\sqrt N}$ for all $n\in\N^+$. For ease of notation, we will denote $\DD^\gamma(\SSS_n,W_n)$ by $\DD^\gamma_n$ for any $\gamma\geq 0$.
    
    If we define
    $$ \gamma_n=16\sqrt N \|C\|_\infty(2\epsilon_n+\delta_n)+L_C\delta_n,$$
    we then have that 
    $$\lim_{n\to\infty}d_{\HH_\HH}\Big(\DD^{\gamma_n}_n,\CC^0\Big)=0,$$
    and
    $$\lim_{n\to\infty}\max_{\B\in\DD^{0}_n}\min_{\B'\in\CC^0}d_\HH(\B,\B')=0.$$

    Furthermore, if there is a unique optimal solution to the continuous problem, i.e.\ $\CC^0=\{\B^\text{CP}\}$ for some $\B^\text{CP}\subset\R^N$, we have
    $$\lim_{n\to\infty}d_{\HH_\HH}\Big(\DD^{0}_n,\CC^0\Big)=0.$$
    \begin{proof}
        We first note that \Cref{prop:CPinDPeps} implies $\CC^0\subset \DD^{\gamma_n}_n$, which yields
        $$d_{\HH_\HH}\Big(\DD^{\gamma_n}_n,\CC^0\Big)=\max_{\B\in\DD^{\gamma_n}_n}\min_{\B'\in\CC^0}d_\HH(\B,\B').$$
        
        Consider therefore any $\B\in\DD^{\gamma_n}_n$. From \Cref{lemm:errorbound_geqC,lem:Bmax<CmaxK,lem:contour_extension_V2} along with \Cref{prop:DPepsinCPeps_continuous}, we get that if we define $\B^{e_n}$ by
        $$\B^{e_n}=\bigcap_{ v \in S^{N-1} } \Pi^-\left( v,B(\B,v) + e_n \right),$$
        where $e_n=\delta_n(L_C+12 \sqrt N (\|C\|_\infty+\gamma_n))$, we have that $d_\HH(\B,\B^{e_n}) = e_n $ and that $\B^{e_n}\in\CC^{\gamma_n+\beta_n}$ for 
        $$\beta_n= 32 \sqrt N (\|C\|_\infty + \gamma_n)\epsilon_n+e_n.$$
        
        As a consequence, we have by the triangle inequality that 
        \begin{align*}
            \min_{\B'\in\CC^0}d_\HH(\B,\B')
            &\leq \min_{\B'\in\CC^0} d_\HH(\B^{e_n},\B') + d_\HH(\B,\B^{e_n})\\
            &\leq  \max_{\B''\in\CC^{\gamma_n+\beta_n}}\min_{\B'\in\CC^0} d_\HH(\B'',\B') + e_n\\
            &=  d_\HH(\CC^{\gamma_n+\beta_n},\CC^0) + e_n.
        \end{align*}
        By noting that $\lim_{n\to\infty}\gamma_n+\beta_n = 0$, we have $d_{\HH_\HH}(\CC^{\gamma_n+\beta_n},\CC^0)\to 0$ by \Cref{prop:CcloseC}. This further yields
        $$\max_{\B\in\DD_n^{\gamma_n}}\min_{\B'\in\CC^0}d_\HH(\B,\B')\to 0,$$
        which yields the first part of the proof. Furthermore, the second equation follows from 
        \begin{align*}
        \max_{\B\in\DD^{0}_n}\min_{\B'\in\CC^0}d_\HH(\B,\B')
        &\leq \max_{\B\in\DD^{\gamma_n}_n}\min_{\B'\in\CC^0}d_\HH(\B,\B')
        \to0
        \end{align*}
        
        Similarly, if $\CC^0=\{\B^\text{CP}\}$,
        \begin{align*}
            d_{\HH_\HH}\left(\DD^{0}_n,\CC^0\right)
            &=\max\left(\max_{\B \in \DD^{0}_n}\min_{\B'\in \CC^0}d_\HH\left(\B,\B'\right), \max_{\B' \in \CC^{0}}\min_{\B \in \DD^{0}_n}d_\HH\left(\B,\B'\right)\right)\\
            &=\max\left(\max_{\B \in \DD^{0}_n}d_\HH\left(\B,\B^\text{CP}\right), \min_{\B \in \DD^{0}_n}d_\HH\left(\B,\B^\text{CP}\right)\right)\\
            &=\max_{\B\in\DD^{0}_n}d_\HH(\B,\B^\text{CP}) \\
            &=\max_{\B\in\DD^{0}_n}\min_{\B'\in\CC^0}d_\HH(\B,\B') 
            \to 0,
        \end{align*}
        which completes the proof.
    \end{proof}
\end{thm}

This theorem tells us that the optimal solution space of the continuous problem can be effectively approximated by a near-optimal solution space of a linear program. Furthermore, if the continuous solution is unique, then the optimal discrete solution space is sufficient to approximate this solution. Lastly, \Cref{thm:dDnC_to0} also proves that all optimal solutions of our discrete problem can be made arbitrarily close to a continuous one in terms of Hausdorff distance. This complements \Cref{thm:errorbound_optval}, and provides an alternate perspective on how our discrete problem approximates the continuous optimal solutions.

\section{Minimal Valid Contours in 2-D}\label{sec:bonus2d}

In the linear program \eqref{lpprob}, we have several constraints in place to ensure that the output corresponds to a convex set. In two dimensions, however, there is a more efficient way of phrasing this property. This relates to the presence of loops in the contour as discussed in e.g.\ \cite{convcont}.

To discuss this, we consider a finite $\SSS=\{u_i\}_{i=1}^m\subset S^1$ and parameterise any $ u\in\SSS $ by the unique angle $ \theta(u)\in [0,2\pi) $ such that $ u=(\cos(\theta),\sin(\theta)) $. Using this parameterisation, we will consider $\SSS$ to be an ordered set $\SSS=(u_i)_{i=1}^m$ such that $\theta(u_i)<\theta(u_{i+1})$. Finally, we also denote $u_{m+1}=u_1$, $u_{0}=u_m$, and $\theta_{i}=\theta(u_i)$ for any $0\leq i \leq m+1$.

In \eqref{lpprob} we have defined our constraints to ensure the $B_i=B(\B,u_i)$ for some convex $\B$, to rephrase those constraints for $N=2$ we define
\begin{equation}\label{eq:BeqCapB}
    \B=\bigcap_{ i=1 }^m \Pi^-\left( u_i,B_i  \right),    
\end{equation}
and examine when $B(\B,u_i)=B_i$ holds. To do so we will denote the hyperplane $\Pi_j=\Pi(u_j,B_j)$ and the crossing point $X_j=\Pi_j\cap\Pi_{j-1}$ for all $j$.

We see from \Cref{fig:noloops} that when $B(\B,u_i)=B_i$, the hyperplane $\Pi_i$ supports $\B$. Furthermore, we can compute the length, $L_i$ of the line segment $\B\cap\Pi_i$, which equals the distance between $ X_i$ and $X_{i-1}$, by 
\begin{align}\label{eq:L_i}
		L_i=&
        \frac{B_{i+1}  -\langle u_{i+1},u_i\rangle B_i  }{\sin(\theta_{i+1}-\theta_{i})}
		+\frac{B_{i-1} -\langle u_{i-1},u_i\rangle B_i  }{\sin(\theta_{i}-\theta_{i-1})}.
\end{align}
%
%We can also draw $\partial\B$ by tracing the points, $\{x_i\}_{i=1}^m$, where $x_i\in \Pi^-\left( u_{i-1},B_{i-1}  \right) \cap \Pi^-\left( u_{i},B_i  \right)$. Performing the same method naively 

The key observation comes from \Cref{fig:loops}, where the resulting contour satisfies $B(\B,u_i) < B_i$, and as such we have $\B\cap\Pi_i=\emptyset$. In particular, we note that $ X_i$ and $X_{i-1}$ switch sides. Since \Cref{eq:L_i} is based on projecting $ X_i$ and $X_{i-1}$ along $\Pi_i$, if we were to compute $L_i$ by \eqref{eq:L_i}, we would see that while it still provides the distance between $X_i$ and $X_{i-1}$, the sign of $L_i$ is now negative. This observation tells us that the condition $B(\B,u_i)=B_i$, for all $i$, is equivalent to $L_i\geq0$ for all $i$.

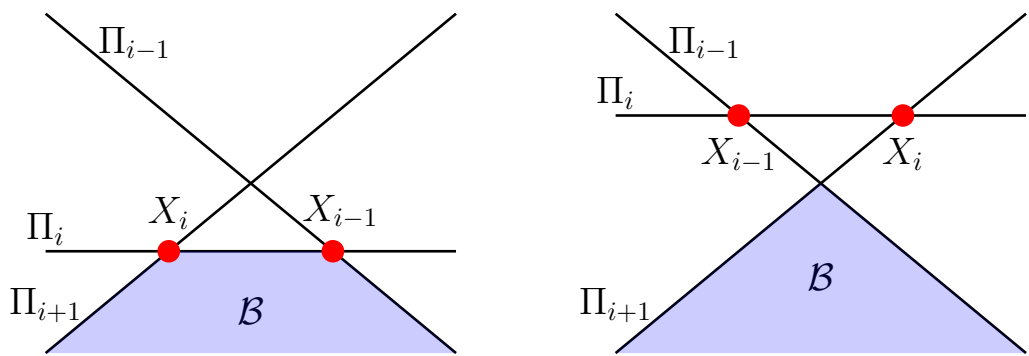
\begin{figure}[h]
	\centering
	\begin{subfigure}{.5\textwidth}
		\centering
		\begin{tikzpicture}[scale=0.9]
			
			%Lines
			
			\coordinate (A) at (0,0);
			\coordinate (B) at (6,5*6/6);
			\draw [name path=A--B,line width=1pt] (A) -- (B);
			
			\coordinate (C) at (6,0);
			\coordinate (D) at (6-6,5*6/6);
			\draw [name path=C--D,line width=1pt] (C) -- (D);
			
			\coordinate (E) at (6,1.5);
			\coordinate (F) at (0,1.5);
			\draw [name path=E--F,line width=1pt,label=$test$] (E) -- (F);
			
			% Colouring Contour Inside
			\path [name intersections={of=A--B and E--F,by=G}];
			\path [name intersections={of=C--D and E--F,by=H}];
			\coordinate (BB) at (3,1);
			\node [fill=none,inner sep=0pt,shape = circle,label=-90:\large$\mathcal{B}$] at (BB) {};
			\draw [ draw=none, fill=blue, opacity=0.2]
			(0,0) -- (G) -- (H) -- (6,0) -- cycle ;
			
			%Intersections
			
			\path [name intersections={of=A--B and E--F,by=G}];
			\node [fill=red,inner sep=3pt,shape = circle,label={[xshift=0cm]+90:\large$X_{i}$}] at (G) {};

			\path [name intersections={of=C--D and E--F,by=H}];
			\node [fill=red,inner sep=3pt,shape = circle,label={[xshift=+0.1cm]\large$X_{i-1}$}] at (H) {};
			
			%Labels for half-spaces
			
			\coordinate (pi1) at (0,1.15);
			\node [fill=none,inner sep=0pt,shape = circle,label=-90:\large$\Pi_{i+1}$] at (pi1) {};
			
			\coordinate (pi2) at (0,2.25);
			\node [fill=none,inner sep=0pt,shape = circle,label=-90:\large$\Pi_{i}$] at (pi2) {};
			
			\coordinate (pi3) at (1.3,5);
			\node [fill=none,inner sep=0pt,shape = circle,label=-90:\large$\Pi_{i-1}$] at (pi3) {};

		\end{tikzpicture}
		
		\caption{Good Case: $B(\B,u_i)=B_i$ and $L_i\geq 0$}
		\label{fig:noloops}
	\end{subfigure}%
	\begin{subfigure}{.5\textwidth}
		\centering
		\begin{tikzpicture}[scale=0.9]
			
			%Lines
			
			\coordinate (A) at (0,0);
			\coordinate (B) at (6,5);
			\draw [name path=A--B,line width=1pt] (A) -- (B);
			
			\coordinate (C) at (6,0);
			\coordinate (D) at (0,5);
			\draw [name path=C--D,line width=1pt] (C) -- (D);
			
			\coordinate (E) at (6,3.5);
			\coordinate (F) at (0,3.5);
			\draw [name path=E--F,line width=1pt,label=$test$] (E) -- (F);
			
			%Intersections
			
			\path [name intersections={of=A--B and E--F,by=G}];
			\node [fill=red,inner sep=3pt,shape = circle,label=-90:\large$X_{i}$] at (G) {};

			\path [name intersections={of=C--D and E--F,by=H}];
			\node [fill=red,inner sep=3pt,shape = circle,label=-90:\large$X_{i-1}$] at (H) {};
			
			%Labels for half-spaces
			
			\coordinate (pi1) at (0,1.15);
			\node [fill=none,inner sep=0pt,shape = circle,label=-90:\large$\Pi_{i+1}$] at (pi1) {};
			
			\coordinate (pi2) at (0,4.25);
			\node [fill=none,inner sep=0pt,shape = circle,label=-90:\large$\Pi_{i}$] at (pi2) {};
			
			\coordinate (pi3) at (1.3,5);
			\node [fill=none,inner sep=0pt,shape = circle,label=-90:\large$\Pi_{i-1}$] at (pi3) {};
			
			% Colouring Contour Inside
			
			\coordinate (BB) at (3,1.5);
			\node [fill=none,inner sep=0pt,shape = circle,label=-90:\large$\mathcal{B}$] at (BB) {};
			\draw [ draw=none, fill=blue, opacity=0.2]
			(0,0) -- (3,2.5) -- (6,0) -- cycle ;
		\end{tikzpicture}
		\caption{Bad Case: $B(\B,u_i) < B_i$ and $L_i < 0$}
		\label{fig:loops}
	\end{subfigure}
	\caption{Two cases for $\B$ (purple), defined by \eqref{eq:BeqCapB}. The condition $\B\cap\Pi_i\neq\emptyset$ fails based on the orientation of $X_i$ and $X_{i-1}$ ({red}). }
	\label{fig:loopnoloop}
\end{figure}

Using this equivalence we can restate our linear program for a quadrature, $(\SSS,W)$, as follows.
\begin{align}\label{lpprob_2d1}
	\text{minimise }   \quad\quad\quad & \sum_{i=1}^m w_iB_i& \\
	\text{subject to } \quad\quad\quad & L_i\geq 0 & i=1,2,\dots,m \nonumber\\
    & B_i\geq C(u_i) & i=1,2,\dots,m \nonumber\\
    &  u_i \in S^{N-1}    & i=1,2,\dots,m \nonumber\\
	&  C(u_i),B_i \in \R    & i=1,2,\dots,m \nonumber
\end{align}	

Here, $L_i$ is defined as in \eqref{eq:L_i}. Furthermore, the restrictions $L_i\geq 0$ and $B_i\geq C(u_i)$ for all $i$ imply that a $\B$ defined by $\B=\bigcap_{ i=1 }^m \Pi^-\left( u_i,B_i  \right)$ must be a convex set satisfying $B(\B,u_i)=B_i\geq C_i$. As such there exists a collection of points $\{p_i\}_{i=1}^m\subset \B$ such that $\langle p_i,u_i\rangle\geq C_i$, implying that $\left(\left(p_i\right)_{i=1}^m , \, \left(B_i\right)_{i=1}^m \right)$ is a feasible solution of our original linear problem \eqref{lpprob}. This, along with our earlier discussion, shows that \eqref{lpprob_2d1} is equivalent to \eqref{lpprob}, but phrased with far less numerically demanding constraints. 

However, due to no longer explicitly storing the points $(p_i)_{i=1}^m$, we are now limited in our explicit construction of sets. We can no longer define $\B^*=  \convh\left(\left\{p^*_i\right\}_{i=1}^m\right)$ from \Cref{prop:Bstar} based on our linear program outputs, and as such exclusively rely on $\B'=  \bigcap_{i=1}^m\Pi^-(u_i,B_i)$. {This construction, as mentioned in \Cref{rem:BstarinBprime} and shown in \Cref{exam:BstarinBprime}, is more conservative and therefore a safer choice than $\B^*$. As such we do not lose much in considering this more efficient method.}

Nevertheless, since $L_i$ represents the length of the $i$'th side this restriction allows us to consider an alternative linear program where we instead minimise $\sum_i L_i$, which equals the circumference of $\B'$. Optimising in mean width and circumference turns out to be equivalent in the case where $\SSS$ is uniformly distributed and $w_i=1/m$ for all $i$. When this is the case we have  $\theta_i-\theta_{i-1}=2\pi/m$ for all $i=2,3,\dots,m$ and $2\pi+\theta_1-\theta_{m}=2\pi/m$, which yields
\begin{align*}
    \sum_{i=1}^m L_i 
    &= \sum_{i=1}^m \frac{B_{i+1}  -\langle u_{i+1},u_i\rangle B_i  }{\sin(\theta_{i+1}-\theta_{i})}
        +\frac{B_{i-1} -\langle u_{i-1},u_i\rangle B_i  }{\sin(\theta_{i}-\theta_{i-1})}\\
    &= \sum_{i=1}^m \frac{B_{i+1}  -\cos(2\pi/m) B_i  }{\sin(2\pi/m)}
    +\frac{B_{i-1} -\cos(2\pi/m) B_i  }{\sin(2\pi/m)}\\
    &=\frac{1}{\sin(2\pi/m)}\left( \sum_{i=1}^mB_{i+1} + \sum_{i=1}^mB_{i-1} -2\sum_{i=1}^m \cos(2\pi/m)B_i\right)\\
    &=(2m)\frac{1-\cos(2\pi/m)}{\sin(2\pi/m)} \sum_{i=1}^mB_{i}/m. \\
\end{align*}
Since any optimal $(B_i)_{i=1}^m$ would minimise $\sum_i w_iB_i$ it must also minimise $\sum_i L_i$, making them equivalent objective functions in this case. For a uniformly distributed $\SSS$, we have that the optimal weights, $W=\{w_i\}_{i=1}^m$, $w_i=1/m$, yield at least a $\frac{\pi}{2m}$-accurate quadrature. Using the quadrature $(\SSS,W)$ will therefore allow us to optimise both mean width and circumference with optimal accuracy. In addition, we can use the efficient formulation of \eqref{lpprob_2d1} which significantly increases the computation speed.

\section{Approximating the mean width}\label{sec:quadraturecomp}
In this section, we approximate the mean width of a convex shape using point samples in the sense of \Cref{def:epsilon_accurate}.

In two dimensions, the simplest example of an $\epsilon$-accurate quadrature is the uniform quadrature $\left\{\Big(\big(\cos\left(\frac{2\pi i}{m}\right), \sin\left(\frac{2\pi i}{m}\right)\big), \frac{1}{m}\Big)\right\}_{i=1}^m \subset S^{N-1}$. To see this, we bound the difference \eqref{eq:epsilon_accurate_def}. By symmetry, we can split the integral into $2m$ identical segments of $\frac{2\pi}{2m}$ radians each, where each part has one endpoint on a quadrature point. Using the fact that $f$ is $L$-Lipschitz, the approximation error is hence at most
\begin{align*}
  \frac{2m}{2\pi} \int_0^{\frac{2\pi}{2m}} L \sqrt{\sin^2(\theta)+(1-\cos(\theta))^2}\,d\theta = \frac{8Lm}{\pi}\sin^2\left(\frac{\pi}{4m}\right) \le \frac{\pi}{2m}L.
\end{align*}
Hence, the uniform quadrature in two dimensions is $\frac{\pi}{2m}$-accurate. It is also straightforward to see that the uniform quadrature has dispersion $2\sin\left(\frac{\pi}{2m}\right) < \frac{\pi}{m}$.

As concrete examples of accurate quadratures in general dimensions, we may use the composite midpoint rule on the cubed hypersphere, which is defined as follows. Let $s \in \N^+$ be the number of subdivisions per dimension. Let $U = \left\{\frac{2i-s-1}{s} \colon i \in \{1,\dots,s\}\right\}$ denote $s$ uniformly distributed points on the segment $[-1,1]$. Then we can define a grid on the faces orthogonal on the $i$th dimension of the hypercube by taking Cartesian products
$$v_i \coloneqq \underbrace{U \times \cdots \times U}_{i-1\text{ times}} \times \{-1,1\} \times \underbrace{U \times \cdots \times U}_{N-i\text{ times}}.$$
The combined grid becomes $V \coloneqq \bigcup_{i = 1}^N v_i$. We number the points in $V$ from $1$ to $m \coloneqq 2Ns^{N-1}$. Then, $\left\{\left(\frac{1}{\norm{V_i}_2}V_i, \frac{1}{\norm{V_i}_2}\left(\frac{2}{s}\right)^{N-1}\right)\right\}_{i=1}^m \subset S^{N-1}$ defines a quadrature which we call the {\it cubed hypersphere quadrature} with $s$ subdivisions.

\begin{prop}\label{cubed_hypersphere_accurate}
  The cubed hypersphere quadrature with $s$ subdivisions is\\
  $\frac{N 2^N\sqrt{N-1}}{s}$-accurate and has dispersion bounded by $\frac{\sqrt{N-1}}{s}$.
\end{prop}
The proof is deferred to the end of this section.

Once we have an $\epsilon$-accurate quadrature, we can transform any set $\{u\}_{i=1}^m$ with small dispersion into an accurate quadrature using the following.

\begin{prop}
  Let $\SSS = \{u_i\}_{i=1}^m \subset S^{N-1}$ have dispersion $\delta = \disp(\SSS)$, and $\{(v_j,z_j)\}_{j=1}^m$ be an $\epsilon$-accurate quadrature. For $j = 1,\dots,m$, let $p_j$ be the index of a point in $\SSS$ closest to $v_j$. Let $p^{-1}_i \coloneqq \left\{j \in \{1,\dots,m\} \colon p_j = i\right\}$ be a set of indices of points in $\{v_j\}_{j=1}^m$ closest to $u_i$. For $i = 1,\dots,m$, let $w_i = \sum_{j \in p^{-1}_i} z_j$. Then $\{(u_i,w_i)\}_{i=1}^m$ is a $\delta+\epsilon(1+\delta)$-accurate quadrature.
\end{prop}
\begin{proof}
  Let $f \colon S^{N-1} \to \R$ be $L$-Lipschitz and have absolute value bounded by $\norm{f}_\infty$, then since $\{(v_j,z_j)\}_{j=1}^m$ is an $\epsilon$-accurate quadrature, we have
  \begin{align}\label{eq:epsilon_accurate}
    \left|\int_S f(u) \,d\sigma(u) - \sum_{j=1}^m f(v_j) z_j\right| \le \epsilon (L+\norm{f}_\infty).
  \end{align}
  Expanding the definitions of $w_i$ and $p^{-1}$, we get
  \begin{align*}
    \sum_{i=1}^m f(u_i) w_i &= \sum_{i=1}^m f(u_i) \sum_{j \in p^{-1}_i} z_j = \sum_{i=1}^m \sum_{j \in p^{-1}_i} f(v_j) z_j + z_j(f(u_i)-f(v_j))\\
                            &= \sum_{j=1}^m f(v_j) z_j + \sum_{j=1}^m z_j(f(u_{p_j})-f(v_j)).
  \end{align*}
   The second term can be bounded by $f$ being Lipschitz continuous, giving
   \begin{align*}
     \left|\sum_{j=1}^m z_j(f(u_{p_j})-f(v_j))\right| \le \sum_{j=1}^m z_j L \norm{u_{p_j}-v_j}_2.
   \end{align*}
   Furthermore, by the definition of $p_j$ as the index of a closest points in $\SSS$, and the definition of $\disp(\SSS)$, we have $\norm{u_{p_j}-v_j}_2 \le \delta$ for $j = 1,\dots,m$. Finally, we can bound $\sum_{j=1}^m z_j \le 1+\epsilon$ by inserting the constant function $f = 1$ into \eqref{eq:epsilon_accurate}.
  
  Combining everything, we get
  \begin{align*}
    \left|\int_S f(u) \,d\sigma(u) - \sum_{j=1}^m f(u_j) w_j\right| &\le \left|\int_S f(u) \,d\sigma(u) - \sum_{j=1}^m f(v_j) z_j\right| + \delta L (1+\epsilon)\\
    \le \epsilon (L+\norm{f}_\infty) + \delta L (1+\epsilon) &\le (\delta + \epsilon(1+\delta))(L+\norm{f}_\infty).
  \end{align*}
  Which is the definition of being a $\delta + \epsilon(1+\delta)$-accurate quadrature.
\end{proof}

\subsection*{Proof of \texorpdfstring{\Cref{cubed_hypersphere_accurate}}{}}

\begin{proof}
  Let $F_i$ denote the $i$th of the $2N$ faces of t he hypercube in $N$ dimensions and let $S_i$  be its projection on the unit hypersphere $S^{N-1}$. Then, 
  \begin{align*}
    \int_{S^{N-1}} f(u) \,d\sigma(u) = \sum_{i=1}^{2N} \int_{S_i} f(u) \,d\sigma(u).
  \end{align*}
  We perform a transformation of variables $u = \frac{1}{\norm{v}_2}v$ to transform the integral over $S_i$ to one over $F_i$.
  \begin{align}\label{eq:face_int}
    \int_{S_i} f(u) \,d\sigma(u) = \int_{F_i} f\left(\frac{1}{\norm{v}_2}v\right) \frac{1}{\norm{v}_2} \,dv.
  \end{align}
  The cubed hypersphere quadrature corresponds to approximating the integral \eqref{eq:face_int} with the composite midpoint rule. That is, the integral is divided into $s^{N-1}$ hypercubes with side lengths $\frac{2}{s}$, each approximated by the value at their midpoint. Let the $j$th of these hypercubes be called $F_{ij}$.
  \begin{align*}
    \int_{F_i} f\left(\frac{1}{\norm{v}_2}v\right) \frac{1}{\norm{v}_2} \,dv = \sum_{j=1}^{s^{N-1}} \int_{F_{ij}} f\left(\frac{1}{\norm{v}_2}v\right) \frac{1}{\norm{v}_2} \,dv.
  \end{align*}
  Note that for all $v \in F_i$, we have $\norm{v}_2 \ge 1$. So since $f$ is $L$-Lipschitz, $f\left(\frac{1}{\norm{v}_2}v\right)$ is also $L$-Lipschitz for $v \in F_i$. Additionally, $\frac{1}{\norm{v}_2}$ is $1$-Lipschitz for $v \in F_i$. Hence, for $u,v \in F_i$,
  \begin{align*}
    &\left|f\left(\frac{1}{\norm{u}_2}u\right) \frac{1}{\norm{u}_2} - f\left(\frac{1}{\norm{v}_2}v\right) \frac{1}{\norm{v}_2}\right|\\
    \le &\left|f\left(\frac{1}{\norm{u}_2}u\right) - f\left(\frac{1}{\norm{v}_2}v\right)\right| \frac{1}{\norm{u}_2} + \left|f\left(\frac{1}{\norm{v}_2}v\right) \left(\frac{1}{\norm{u}_2}-\frac{1}{\norm{v}_2}\right)\right|\\
    \le &\norm{u-v}_2 (L+\norm{f}_\infty).
  \end{align*}
  So $f\left(\frac{1}{\norm{v}_2}v\right) \frac{1}{\norm{v}_2}$ is Lipschitz over $F_i$ with constant $L+\norm{f}_\infty$, where $\norm{f}_\infty$ is an upper bound on the absolute value of $f$. Consider the $j$th hypercube, let its midpoint be $V_j$. Inside the hypercube, the maximum distance to the midpoint is $\frac{\sqrt{N-1}}{s}$, so we get an error bounded by $\frac{\sqrt{N-1}}{s}(L+\norm{f}_\infty)$. Specifically,
  \begin{align*}
    \left|\int_{F_{ij}} f\left(\frac{1}{\norm{v}_2}v\right) \frac{1}{\norm{v}_2} \,dv - \int_{F_{ij}} f\left(\frac{1}{\norm{V_j}_2}V_j\right) \frac{1}{\norm{V_j}_2} \,dv\right|\\
    \le \int_{F_{ij}} \frac{\sqrt{N-1}}{s}(L+\norm{f}_\infty) \,dv.
  \end{align*}
  Next, note that the volume of $F_{ij}$ is $\left(\frac{2}{s}\right)^{N-1}$, so
  \begin{align*}
    \int_{F_{ij}} f\left(\frac{1}{\norm{V_i}_2}V_i\right) \frac{1}{\norm{V_i}_2} \,dv = f\left(\frac{1}{\norm{V_i}_2}V_i\right)\ \frac{1}{\norm{V_i}_2}\left(\frac{2}{s}\right)^{N-1} = f(u_i) w_i,
  \end{align*}
  where $u_i$ and $w_i$ are the cubed hypersphere quadrature. Hence the $i$th hypercube contributes an error at most
  \begin{align*}
    \left|\int_{F_{ij}} f\left(\frac{1}{\norm{v}_2}v\right) \frac{1}{\norm{v}_2} \,dv - f(u_i)w_i\right| \le \frac{\sqrt{N-1}}{s}(L+\norm{f}_\infty)\left(\frac{2}{s}\right)^{N-1}.
  \end{align*}
  Summing over the $m = 2Ns^{N-1}$ hypercubes, we get an error at most
  \begin{align*}
    \left|\int_S f(u) \,d\sigma(u) - \sum_{i=1}^m f(u) w_i\right| &\le 2Ns^{N-1} \frac{\sqrt{N-1}}{s}(L+\norm{f}_\infty)\left(\frac{2}{s}\right)^{N-1}\\
                                                        &= \frac{N 2^N\sqrt{N-1}}{s}(L+\norm{f}_\infty).
  \end{align*}
  Hence the quadrature is $\frac{N 2^N\sqrt{N-1}}{s}$-accurate.

  The dispersion can also be bounded by considering the hypercubes $F_{ij}$. Again, all points in the $j$th hypercube of the $i$th face are within a distance $\frac{\sqrt{N-1}}{s}$ from the midpoint. Distances in the hypercube only get smaller when they are projected onto the unit hypersphere $S^{N-1}$, and the projected midpoints are exactly the cubed hypersphere quadrature points. Hence, the dispersion is bounded by $\frac{\sqrt{N-1}}{s}$.
\end{proof}

\section{Conclusion and Final Remarks}

In conclusion, this paper has presented a novel algorithm for the computation of valid environmental contours. The proposed algorithm ensures that the contours satisfy the outreach requirements while maintaining a minimal mean width. We have also presented a streamlined algorithm for two dimensions, which improves computation speed for this specific case. Both of the considered methods have been illustrated by numerical examples.

Furthermore, as these methods rely on numerical integration, we also provided a generic {construction for making} arbitrarily accurate quadratures. 

Lastly, rigorous examination of convergence and existence of solutions has been conducted, ensuring the reliability and accuracy of the proposed methods. Convergence properties have been thoroughly analyzed, including the convergence of the optimal approximate mean width to the optimal mean width, as well as convergence in terms of the Hausdorff metric. These analyses ensure that any approximate solution will give an arbitrarily near-optimal contour, and that any optimal contour can be found by searching the near-optimal approximations.

\section*{Acknowledgements}
The authors acknowledge financial support by the Research Council of Norway under the SCROLLER project, project number 299897 (Åsmund Hausken Sande).

\bibliographystyle{abbrvdin}
\bibliography{refs}

%\begin{thebibliography}{99}

%\bibitem{Hausdorffproperties}Price, G. Baley. "On the completeness of a certain metric space with an application to Blaschke's selection theorem." (1940): 278-280.

%\bibitem{norsok}NORSOK N-003, 2017. NORSOK Standard N-003:2017: Actions and Action Effects. NORSOK, Norway % too expensive to cite

%\bibitem{convcont}Convex environmental contours

%\bibitem{voronoi}Environmental contours as Voronoi cells

%\bibitem{altcontour}Alternative environmental contours for structural reliability analysis

%\bibitem{firstaltcontour}A new approach to environmental contours for ocean engineering applications based on direct Monte Carlo simulations
%\end{thebibliography}

\end{document}